\documentclass[11pt,letterpaper]{article}

\usepackage[T1]{fontenc}
\usepackage{lmodern}
\usepackage[margin=1in]{geometry}
\usepackage{amsmath}
\usepackage[round]{natbib}
\usepackage{microtype}
\newcommand{\cF}{\mathcal{F}}
\newcommand{\cH}{\mathcal{H}}
\newcommand{\cI}{\mathcal{I}}

\newcommand{\cP}{\mathcal{P}}

\newcommand{\N}{\mathbb{N}}

\usepackage{algorithm}
\usepackage{algpseudocode}
\usepackage{appendix}
\usepackage{amssymb}
\usepackage{amsthm}
\usepackage{mathtools}
\usepackage{booktabs}
\usepackage{threeparttable}
\usepackage{bbm}
\usepackage{tabularx}
\usepackage{todonotes}
\usepackage{subcaption}
\usepackage{graphicx}
\usepackage{xcolor}
\usepackage{hyperref}
\usepackage{cleveref}

\newcommand{\E}{\mathbb{E}}

\newcommand{\R}{\mathbb{R}}

\newcommand{\ltwonorm}[1]{\left\| #1\right\|}
\newcommand{\set}[1]{\left\{#1\right\}}

\def\rvx{{\mathbf{x}}}
\def\rvy{{\mathbf{y}}}

\newcommand{\bracket}[1]{\left( #1 \right)}
\newcommand{\sqbk}[1]{\left[ #1 \right]}

\newtheorem{theorem}{Theorem}
\newtheorem{assumption}{Assumption}
\newtheorem{proposition}{Proposition}
\newtheorem{lemma}{Lemma}
\newtheorem{corollary}{Corollary}

\hypersetup{
  hidelinks,
  pdftitle={Large-Step Last-Iterate Convergence of Random-Reshuffling Stochastic Gradient Descent},
  pdfauthor={Zijun Chen and Zihan Zhang}
}

\title{Large-Step Last-Iterate Convergence of Random-Reshuffling Stochastic Gradient Descent}
\author{
  Zijun Chen\\
  \normalsize Department of Computer Science and Engineering\\
  \normalsize Hong Kong University of Science and Technology\\
  \normalsize\texttt{zchendg@connect.ust.hk}
  \and
  Zihan Zhang\thanks{Corresponding author}\\
  \normalsize Department of Computer Science and Engineering\\
  \normalsize Hong Kong University of Science and Technology\\
  \normalsize\texttt{zsubfunc@outlook.com}
}
\date{}

\begin{document}
\maketitle

\begin{abstract}
Random reshuffling stochastic gradient descent (RR-SGD) is widely used for finite-sum optimization in machine learning. For convex $L$-smooth components, recent work establishes convergence guarantees for an averaged iterate with component stepsizes as large as order $1/L$~\citep{liu2026random}. However, the final iterate is a natural output in practice, and existing last-iterate guarantees in this setting require stepsizes that scale inversely with the number of components $n$.  We establish an expected objective-gap bound for the last iterate of standard RR-SGD on unconstrained problems with a constant component stepsize satisfying $0<\eta\lesssim 1/(L(1 + \log(nK)))$, where $K$ is the number of epochs. 
By retaining the sampling bias at random comparison points, we remove the explicit $1/n$ factor from the stepsize restriction. Combined with existing small-step bounds, our result matches the corresponding averaged-iterate convergence rate up to logarithmic factors.
\end{abstract}

\section{Introduction}
\label{sec:introduction}

Finite-sum optimization is a basic problem in empirical risk minimization, where a learning objective averages losses over a fixed dataset. We consider
\begin{equation}
\min_{\rvx\in\mathbb R^d} f(\rvx),
\qquad f(x)=\frac{1}{n}\sum_{i=1}^n f_i(\rvx),
\label{eq:intro-problem}
\end{equation}
with convex, $L$-smooth component functions and an attained minimum. When $n$ is large, evaluating a single component gradient can be much cheaper than computing the full gradient. Stochastic gradient descent (SGD) exploits this structure by sampling a component independently at each update. Random reshuffling (RR) instead visits every component once per epoch, drawing a fresh independent uniform permutation for each pass through the data. This sampling rule is common in training implementations and has motivated extensive work on optimization without replacement~\citep{bottou2009curiously, shamir2016without}

After a prescribed number of epochs, the final parameter is a natural output of the algorithm. A convergence bound for an average of earlier iterates does not, however, directly describe this output. For RR, the distinction is complicated by the dependence between the current iterate and the components remaining in an epoch: their gradients are generally biased estimates of the full gradient. Understanding the final iterate therefore requires accounting for the sampling rule throughout the optimization trajectory.

Recent work has established last-iterate guarantees for smooth convex shuffling methods using component-gradient information at a minimizer \citep{liu2024last,cai2025last}. In their plain-RR specializations, the stepsize multiplying each component gradient is restricted to order $1/(nL)$, with an additional logarithmic factor. By contrast, \citet{liu2026random} allows component stepsizes of order $1/L$ under the same convex smoothness assumptions, but controls an averaged output. Its bound combines the usual SGD noise scale with a sharper term available from reshuffling.
The averaged-output result shows that the stepsize upper bound can depend only on the smoothness parameter $L$, independently of the number of components $n$. This naturally raises the question:
\emph{Under the same convexity and smoothness assumptions, can we remove the explicit $1/n$ factor from the stepsize restriction for last-iterate convergence?}

\subsection{Our Contribution}
\label{sec:our_contribution}
We answer this question in the affirmative, up to logarithmic factors. For a run of $K$ epochs, we control the last iterate endpoint $\rvx_{K+1}^{1}$\footnote{Theorem~\ref{thm:main} also covers every prescribed component-time iterate $\rvx_k^i$ in the stated horizon. We focus on the endpoint $\rvx_{K+1}^1$ in the introduction.}. For smooth convex optimization (i.e., each $f_i$ is convex and $L$-smooth), Theorem~\ref{thm:main} states that RR-SGD with a constant stepsize $\eta$ satisfies that
\begin{align}
&\mathbb E[f(\rvx_{K+1}^{1})-f(\rvx_{\star})]\lesssim
\frac{D^2}{\eta nK}
+\min\{\eta\sigma_\star^2,\eta^2nL\sigma_\star^2\}
. \nonumber
\end{align}
provided that the stepsize satisfies $\eta \leq c/(L(1 + \log (nK)))$ for a universal constant $c>0$. Here, $\lesssim$ only hides the universal constants and logarithmic factors in $nK$. $D = \|\rvx_1^1 - \rvx_{\star}\|$ ($\|\cdot\|$ is the $\ell_2$-norm) is the distance from the initial point to a minimizer $\rvx_{\star}$, and $\sigma_{*}^{2}$ is the gradient variance at the minimizer, which is invariant even when $f$ has multiple minimizers.

The first term, $D^2/(\eta nK)$, depends on the initial distance from the optimum and decreases with the number of updates for a fixed stepsize. The second term, $\min\{\eta\sigma_\star^2,\eta^2nL\sigma_\star^2\}$, accounts for the uncertainty induced by shuffling. The two expressions in this minimum describe the same stepsize regimes as the averaged-output bound of \citet{liu2026random}: $\eta^2nL\sigma_\star^2$ is smaller when $\eta nL\leq 1$, whereas $\eta\sigma_\star^2$ is smaller when $\eta nL\geq 1$. For $D>0$, specializing to the epoch endpoint after $T=nK$ updates and choosing a constant stepsize using the horizon and problem parameters recovers the polynomial parameter dependence of the corresponding tuned averaged guarantee, up to logarithmic factors (Corollary~\ref{cor:tuned}). When $\sigma_\star=0$, taking $\eta=c/[L(1+\log T)]$ yields an endpoint gap of $\widetilde O(LD^2/T)$.

\textbf{Technical contribution.}
The large-step analysis of \citet{liu2026random} controls the sampling bias (see definition at~\cref{equ:sampling_bias}) by comparing iterates with a fixed minimizer and summing over complete epochs.
Extending this approach to the last iterate requires analyzing suffixes that may begin within an epoch.
The starting iterate of each suffix serves as a random comparison point and may be correlated with the components remaining in that epoch.
This dependence introduces an additional bias that must be controlled jointly with the sampling bias along the trajectory.
To address this issue, our analysis combines two key ingredients:
\begin{itemize}
\item \textbf{Signed sampling bias cancellation.}
We track the sampling bias contributions both along the trajectory and at the random comparison point used in each suffix. The contribution at the comparison point equals the negative of the sampling bias at the suffix's starting position throughout the remainder of that epoch and vanishes in subsequent epochs. Retaining these signs through the suffix-to-last conversion reveals a partial cancellation between the two contributions. The resulting weights are nonnegative and satisfy a within-epoch summation bound, allowing us to apply a one-sided sampling-bias estimate and control repeated contributions from earlier updates without an additional factor of $n$.

\item \textbf{A gradient-moment recursion for last-iterate analysis.}
Building on the permutation-swap tools of \citet{liu2026random}, we derive a self-bounding estimate that relates the sampled-gradient second moment to the current objective gap, the gradient variance at a minimizer, and earlier gradient moments within the same epoch. The history term carries a coefficient of order $\eta^2L^2/n$, allowing its accumulated contribution under the suffix-to-last weights to be absorbed with only logarithmic overhead. Together with the signed sampling bias cancellation, this estimate yields a harmonic recursion in the objective gaps that preserves the inverse-time decay of the initial-error term and closes under $\eta L(1+\log(nK))\lesssim 1$, without an explicit inverse-$n$ restriction on the stepsize.
\end{itemize}

\section{Related Work}
\label{sec:related-work}
Table~\ref{tab:related_rr_rates} compares the closest smooth convex RR guarantees, including their admissible component stepsizes and output choices.

\begin{table*}[t]
\centering
\begin{threeparttable}

\caption{Comparison of expected objective-gap bounds for RR-SGD
after $K\geq 2$ epochs with a constant component stepsize $\eta$.
All results are stated for unconstrained finite-sum optimization
with convex, $L$-smooth components and an attained minimum.
Here $D=\|\rvx_1^1-\rvx_*\|$,
$\sigma_*^2=\frac{1}{n}\sum_{i=1}^n
\|\nabla f_i(\rvx_*)\|^2$,
$\ell_K=1+\log K$, and $\Lambda=1+\log(nK)$.
We use a common smoothness bound $L$; the original results may
retain finer smoothness parameters.
$O(\cdot)$ hides universal constants, while
$\widetilde{O}(\cdot)$ additionally hides logarithmic factors.
The average-iterate and last-iterate bounds control
$\E[f(\bar{\rvx})-f(\rvx_*)]$ and
$\E[f(\rvx_{K+1}^1)-f(\rvx_*)]$, respectively,
where $\bar{\rvx}$ denotes the averaged output used in
the corresponding reference.}
\label{tab:related_rr_rates}

\begin{tabular*}{\textwidth}{@{\extracolsep{\fill}}llll@{}}
\toprule
\textbf{Stepsize} & \textbf{Reference}
& \textbf{Rate} & \textbf{Output} \\
\midrule
$\eta\lesssim 1/(nL)$
& \citet{CaiLinDiakonikolas2024PrimalDual}
& $O\!\left(\frac{D^2}{\eta nK}+\eta^2nL\sigma_*^2\right)$
& Average iterate \\

$\eta\lesssim 1/(nL\sqrt{\ell_K})$
& \citet{liu2024last}
& $\widetilde{O}\!\left(
\frac{D^2}{\eta nK}+\eta^2nL\sigma_*^2\right)$
& Last iterate \\

$\eta\lesssim 1/(nL\sqrt{\ell_K})$
& \citet{cai2025last}
& $\widetilde{O}\!\left(
\frac{D^2}{\eta nK}+\eta^2nL\sigma_*^2\right)$
& Last iterate \\

$\eta\lesssim 1/L$
& \citet{liu2026random}
& $O\!\left(\frac{D^2}{\eta nK}
+\sigma_*^2\min\{\eta,\eta^2nL\}\right)$
& Average iterate \\
\midrule
$\eta\lesssim 1/(L\Lambda)$
& \textbf{This work (Theorem~\ref{thm:main})}\tnote{1}
& $\widetilde{O}\!\left(\frac{D^2}{\eta nK}
+\sigma_*^2\min\{\eta,\eta^2nL\}\right)$
& Last iterate \\
\bottomrule
\end{tabular*}

\begin{tablenotes}[flushleft]
\footnotesize
\item[1] Theorem~\ref{thm:main} additionally controls every
prescribed component-time iterate, with $nK$ replaced by
$\tau(k,i)$ in the initial-distance term.
\end{tablenotes}

\end{threeparttable}
\end{table*}

\paragraph{Convergence of shuffling methods.}
Early incremental analyses established convergence under cyclic and related sampling rules \citep{nedic2001incremental}, while \citet{recht2012toward} connected without-replacement iterations to matrix-product inequalities. Under strong convexity, \citet{gurbuzbalaban2021random} explained asymptotic advantages of RR, and \citet{haochen2019random,rajput2020closing} obtained increasingly sharp finite-epoch guarantees. Related analyses characterize cyclic incremental methods \citep{gurbuzbalaban2019convergence} and the steady-state behavior of constant-step reshuffling \citep{ying2018stochastic}. \citet{ahn2020sgd} weakened component assumptions under a Polyak--\L{}ojasiewicz condition and removed large-epoch restrictions in a strongly convex result using varying steps. %These results clarify the benefits of reshuffling under curvature assumptions stronger than convexity alone.

\paragraph{Smooth convex objectives.}
For general smooth convex components, \citet{nagaraj2019sgd} obtained an averaged-output guarantee with component steps as large as order $1/L$, assuming bounded gradients on a bounded domain. \citet{mishchenko2020random,nguyen2021unified} obtained averaged guarantees based on component gradients at a minimizer, with component steps of order $1/(nL)$. \citet{CaiLinDiakonikolas2024PrimalDual} refined the dependence on smoothness and data geometry through a primal-dual analysis, again for an averaged output. The large-step bound of \citet{liu2026random} combines the SGD and reshuffling noise regimes without a global variance assumption. 

%Our result recovers its polynomial dependencies at the final iterate, up to logarithms, in the common-smoothness, constant-step setting.

\paragraph{Last-iterate convergence.}
\citet{liu2024last} proved objective-gap guarantees for the last iterate of proximal shuffling methods, including plain RR when the regularizer is zero. \citet{cai2025last} established related endpoint guarantees for incremental gradient and proximal methods. Their smooth convex RR bounds use a stepsize ceiling of order $1/(nL\sqrt{\log K})$ for $K\geq2$, with an explicit $1/n$ factor. \citet{LiuZhou2025NonsmoothShuffling} studied outputs at arbitrary component times for convex, Lipschitz components, a different assumption from smoothness with potentially unbounded gradients. Classical last-iterate SGD analyses \citep{shamir2013last,harvey2019tight,jain2019last} supply useful techniques, but independent unbiased sampling avoids the within-epoch dependence present in RR.~\citet{attia2025fastlastiterateconvergencesgd} established
last-iterate guarantees for without-replacement SGD with
convex, $L$-smooth components under interpolation
($\sigma_*^2=0$). Their without-replacement analysis is
restricted to iterates within a single pass without replacement and does not cover nonzero noise at the optimum.
In contrast, we establish last-iterate guarantees for
multi-epoch random reshuffling beyond interpolation,
allowing $\sigma_*^2>0$ and capturing both the SGD-scale
and reshuffling-specific noise terms.
%These results establish that large component steps are compatible with convex smooth optimization, but they do not directly control the endpoint returned by standard RR.

\paragraph{Lower bounds and epoch dependence.}
\citet{safran2020shuffling,safran2021epochs} showed that the advantage of reshuffling depends on the function class, conditioning, and number of epochs. For convex objectives, \citet{cha2023lower} established a lower bound at the reshuffling scale for sufficiently many epochs, a restricted small-step range, and a global gradient-deviation parameter. That result does not settle optimality throughout the larger stepsize range considered here, or the necessity of our logarithmic factors. The small-epoch lower bounds of \citet{KimChaYun2025SmallEpochIGD} concern deterministic cyclic IGD and do not directly apply to fresh RR. Related comparisons between reshuffling and reuse of a single permutation also depend on the problem class \citep{yun2021single}.

\section{Preliminary}
\label{sec:preliminary}

\paragraph{Notation.} Let $\N$ denote the set of natural numbers. Given $n\in \N$, we write $[n]:=\{1, \ldots, n\}$, $\langle \cdot , \cdot \rangle$ is the standard Euclidean inner product, and $\|\cdot\|:=\sqrt{\langle \cdot , \cdot \rangle}$ is the $\ell_2$-norm. Given a real-valued differentiable function $h:\R^{d}\to \R$, $\nabla h(\rvx)$ denotes the gradient at $\rvx \in \R^{d}$.  A permutation $\pi$ of $[n]$ is a bijective mapping from $[n]$ to itself. Let $S_n$ denote the set of all permutations of $[n]$. We represent each permutation as $\pi=(\pi^1,\ldots,\pi^n)$, where $\pi^i$ denotes the index assigned to position $i$.

\paragraph{Objective.} We consider the finite-sum optimization problem
\[
\min_{\rvx \in \R^{d}}f(\rvx),\quad f(\rvx) = \frac{1}{n}\sum_{i=1}^{n}f_{i}(\rvx), \quad \rvx \in \R^{d}
\]
and $n\in \N$ and for each $f_i:\R^{d}\to \R$ is convex differentiable. Assume $f$ attains its minimum. Fix $\rvx_{*}\in \arg\min_{\rvx \in \R^{d}}f(\rvx)$ with a deterministic initialization $\rvx_{1}^{1}$.

Our analysis relies on the following key Assumptions

\begin{assumption}[Minimizer]
    \label{ass:minimizer} There exists a minimizer $\rvx_{*}\in \R^{d}$ such that $f(\rvx_{*})=\inf_{\rvx \in \R^{d}}f(\rvx)$.
\end{assumption}

\begin{assumption}[Convexity]
    \label{ass:convexity} For all $i\in [n]$, $f_{i}$ is convex.
\end{assumption}

\begin{assumption}[Smoothness]
    \label{ass:smoothness}
Each $f_{i}$ is $L_{i}$-smooth, i.e., $\forall \rvx, \rvy\in\R^{d}$,
\[
\ltwonorm{\nabla f_{i}(\rvx) - \nabla f_{i}(\rvy)} \leq L_{i} \ltwonorm{\rvx - \rvy}.
\]
Denote $L:= \max_{i\in [n]} L_i$.
We assume $L>0$.
\end{assumption}
All of the three assumptions are standard in the literature on stochastic optimization. We introduce two notations to simplify the expressions in the subsequent sections
\[
D:= \ltwonorm{\rvx_{1}^{1} - \rvx_{*}},\qquad \sigma_*^2 :=\frac{1}{n}\sum_{i=1}^{n}\ltwonorm{\nabla f_{i}(\rvx_{*})}^2
\]
We note that $D$ is the distance between the initial point and the optimal solution. $\sigma_{*}$ is commonly used to measure the noise level of the stochastic gradients at the optimum $\rvx_{*}$, which is invariant even when $f$ has multiple minimizers. In particular, $\sigma_{*}^{2}=0$ corresponds to the case in which all $f_{i}$ have a same optimal solution.

\section{RR-SGD and Main Result}
\label{sec:rr_sgd_and_main_result}
\begin{algorithm}[ht]
\caption{Random-Reshuffling Stochastic-Gradient-Descent (RR-SGD)}
\label{alg:rr_sgd}
\begin{algorithmic}[1]
{
\State {\textbf{input:} initial point $\rvx_{1}^{1}\in\R^{d}$, stepsize $\{\eta_k > 0\}$}
\For{$k=1,\ldots, K$}
\State {Independently draw a permutation $\pi_k\sim S_n$}
\For{$i=1, \ldots, n$}
\State {$\rvx_{k}^{i+1}=\rvx_{k}^{i} - \eta_k \nabla f_{\pi_{k}^{i}}(\rvx_{k}^{i})$}
\EndFor
\State {$\rvx_{k+1}^{1}=\rvx_{k}^{n+1}$}
\EndFor
}
\end{algorithmic}
\end{algorithm}

Algorithm~\ref{alg:rr_sgd} describes random reshuffling stochastic gradient descent (RR-SGD). At the start of epoch $k$, RR-SGD draws a fresh uniform permutation $\pi_k\in S_n$, independently of earlier epochs. It then visits the components in this order: at position $i$, it evaluates $\nabla f_{\pi_k^i}(\rvx_k^i)$ and takes a gradient step with stepsize $\eta_k$. Each component is therefore used exactly once per epoch. The endpoint of epoch $k$ becomes the starting point of the next epoch, $\rvx_{k+1}^1=\rvx_k^{n+1}$. We henceforth consider the constant stepsize $\eta_k\equiv\eta$ for all $k\in[K]$. The output of the algorithm is the final iterate $\rvx_{K+1}^{1}$ after $K$ epochs.

\subsection{Last-Iterate Convergence Guarantee}
We now state our last-iterate guarantee. The stepsize is chosen before the run and remains constant across all component updates. Theorem~\ref{thm:main} controls the iterate at any prescribed component time, including the endpoint $\rvx_{K+1}^{1}$ returned by Algorithm~\ref{alg:rr_sgd}. Write 
\[
\tau(k,i)=(k-1)n+i,\qquad i\in [n]
\]
so that $\rvx_k^i$ is the iterate after $\tau(k,i)-1$ updates.
\begin{theorem}
\label{thm:main}
    
    Let $K\geq1$ and $\Lambda=1+\log(nK)$. Under Assumptions~\ref{ass:minimizer},~\ref{ass:convexity}, and~\ref{ass:smoothness}, run RR-SGD (Algorithm~\ref{alg:rr_sgd}) with $\eta_k\equiv\eta$ for $k\in[K]$, where $0<\eta\leq 1/(1024L\Lambda)$. For every prescribed pair $(k,i)$ with $i\in[n]$ and $2\leq\tau(k,i)\leq nK+1$, the iterate $\rvx_k^i$ satisfies
    \begin{equation*}
    \begin{aligned}
        &\E\sqbk{f(\rvx_k^i)-f(\rvx_*)}\\
        &\quad \leq O\bracket{\frac{D^2}{\eta\tau(k,i)}\nonumber+\Lambda^{3/2}\min\set{\eta\sigma_*^2,\eta^2nL\sigma_*^2}} .\nonumber
    \end{aligned}    
    \end{equation*}

\end{theorem}
\paragraph{Admissible stepsize.} Theorem~\ref{thm:main} allows a constant component stepsize of order $1/(L\log(nK))$, without an explicit inverse dependence on the epoch length $n$. To the best of our knowledge, this is the first last-iterate guarantee for standard RR-SGD over any number of epochs that allows nonzero noise $\sigma_{*}^{2}>0$ and removes the inverse-$n$ factor from the component-stepsize restriction. The existing small-step last-iterate guarantees discussed above impose a component stepsize ceiling of order $1/(nL)$~\cite{liu2024last}. This distinction matters for the initial-error term {$D^2/(\eta\tau(k,i))$}: a larger admissible stepsize allows this term to decrease more rapidly when the noise is sufficiently small. {The guarantee applies at every prescribed component time; the epoch endpoint is the case $(k,i)=(K+1,1)$.}

\paragraph{Two noise regimes.} The minimum in the bound captures two complementary regimes. Since $\eta^2 nL\sigma_\star^2=(\eta nL)\eta\sigma_\star^2$, the reshuffling term $\eta^2 nL\sigma_\star^2$ is smaller when $\eta nL\leq 1$, whereas the SGD-scale term $\eta\sigma_\star^2$ is smaller when $\eta nL\geq 1$. Thus, the theorem retains the sharper reshuffling dependence for small stepsizes while extending last-iterate control to a substantially larger stepsize range. Our new analysis establishes the large-step estimate; the two-regime guarantee follows by combining it with an existing small-step endpoint bound and a within-epoch comparison (Lemma~\ref{lem:within_epoch_comparison}).

% To the best of our knowledge, in smooth convex optimization, Theorem~\ref{thm:main} shows the first last-iterate convergence rate for RR-SGD with large-stepsize, i.e., $\eta \lesssim 1/(L\log (nK)$ rather than carrying an explicit $1/n$ factor. Previously,~\citet{liu2026random} shows the averaged optimal rate admits $\tfrac{LD^2}{nK} +\min\{\frac{\sigma_*D}{\sqrt{nK}}, (\tfrac{L\sigma_{*}^{2} D^4}{nK^2})^{\tfrac{1}{3}}\}$, where~\citet{liu2024last} established the same convergence rate for the last-iterate only for small stepsize $\eta_k\lesssim 1/(nL)$.

\begin{corollary}
    \label{cor:tuned}
    {
    Under the assumptions of Theorem~\ref{thm:main}, let $D>0$ and $\Lambda=1+\log(nK)$. For $\sigma_*>0$, define
    \begin{align*}
        \eta_{\mathrm{lin}}&:=\frac{D}{\Lambda^{3/4}\sigma_*\sqrt{nK}},\\
        \eta_{\mathrm{quad}}&:=\left(\frac{D^2}{\Lambda^{3/2}n^2LK\sigma_*^2}\right)^{1/3},
    \end{align*}
    and choose the constant stepsize
    \[
        \eta=\min\left\{\frac{1}{1024L\Lambda},\max\{\eta_{\mathrm{lin}},\eta_{\mathrm{quad}}\}\right\}.
    \]
    When $\sigma_*=0$, set $\eta=1/(1024L\Lambda)$. In either case, RR-SGD satisfies
    }
    \begin{align}
        &\E\sqbk{f(\rvx_{K+1}^{1}) - f(\rvx_{*})}\nonumber \\
        &\quad\leq  \widetilde O\bracket{\frac{LD^2}{nK} +\min\set{\frac{\sigma_*D}{\sqrt{nK}}, \bracket{\frac{L\sigma_{*}^{2}D^4}{nK^2}}^{1/3}}}
    \end{align}
\end{corollary}
The candidates $\eta_{\mathrm{lin}}$ and $\eta_{\mathrm{quad}}$ balance the initial-distance term against the two expressions in the noise minimum in Theorem~\ref{thm:main}, respectively. Taking their maximum gives the smaller of the two balanced bounds; truncating the stepsize at $1/(1024L\Lambda)$ adds at most a term of order $\Lambda LD^2/(nK)$. Substitution into Theorem~\ref{thm:main} therefore gives Corollary~\ref{cor:tuned}. Up to logarithmic factors, this recovers the polynomial parameter dependence of the corresponding averaged-output guarantee of \citet{liu2026random}. In particular, when $\sigma_\star=0$, choosing the largest admissible stepsize gives $\widetilde{O}(LD^2/T)$ after $T=nK$ component-gradient evaluations. The general guarantee also permits $\sigma_\star>0$  and requires no global bound on gradients or their variance.

%\section{Technique Overview} \label{sec:technique_overview}

\section{Theoretical Proof}
\label{sec:theoretical_proof}
We prove Theorem~\ref{thm:main} in four steps. First, we use permutation-swap stability to derive self-bounding estimates for the sampling bias and the sampled-gradient second moment in Section~\ref{subsec:technical_proof_part_1}. Second, we retain the bias at random comparison points throughout a suffix-to-last conversion, obtaining weights that control repeated within-epoch contributions in Section~\ref{subsec:technical_proof_part_2}. Third, we combine these estimates and close a harmonic-weight recursion to establish the large-step last-iterate guarantee in Section~\ref{subsec:technical_proof_part_3}. Finally, we combine this guarantee with an existing small-step bound {and a within-epoch comparison} to obtain the two noise regimes in Theorem~\ref{thm:main} in Section~\ref{subsec:putting_together}. 

We define the auxiliary quantities
\begin{align}
    \Delta_{k}^{i} :=& \E\sqbk{f(\rvx_{k}^{i})- f(\rvx_{*})}\nonumber\\
    \beta_{k}^{i} :=& \E\sqbk{f(\rvx_{k}^{i}) - f_{\pi_{k}^{i}}(\rvx_{k}^{i})}\label{equ:sampling_bias}
\end{align}

{Here, $\beta_{k}^{i}$ is the \emph{sampling bias}. The challenge is not just that $\beta_{k}^{i}$ may be nonzero: suffix comparisons also introduce bias at an earlier random iterate. For a suffix starting at $\rvx_{p}^{q}$, this comparator bias equals $-\beta_{p}^{q}$ in epoch $p$ and vanishes thereafter (Lemma~\ref{lem:exact_comparator_bias_identify}). We combine the two signed contributions before bounding them. Their partial cancellation yields nonnegative bias weights with controlled within-epoch sums, allowing the one-sided estimate in~\cref{equ:4.10} to be used in the last-iterate bound. Together with the sampled-gradient self-bound, this produces a recursion that closes under $\eta L(1 +\log (nK))\lesssim 1$ without requiring $\eta nL \lesssim 1$.}

% \textcolor{red}{Notably, the sampling bias $\beta_k^i$ in~\cref{equ:sampling_bias} accounts for the dependence between the sampled index $\pi_k^i$ and the current iterate $\rvx_k^i$. If these variables were independent, the uniform distribution of $\pi_k^i$ on $[n]$ would give $\beta_k^i=0$. Under RR-SGD, however, the previously sampled indices influence $\rvx_k^i$ and restrict the components available at position $i$, so this independence generally fails and $\beta_k^i$ need not vanish. The key reduction in our last-iterate analysis is to express all function-value bias terms in the suffix-to-last argument, including those at random comparison points, through this same quantity.}

% \textcolor{red}{Our key improvement is to retain the signed bias at random comparison points throughout the suffix-to-last conversion. The same $\beta_k^i$ enters the trajectory-bias and comparison-point-bias terms with coefficients of opposite signs. Combining these contributions before bounding them reveals a partial cancellation. The resulting bias weights are nonnegative and have controlled sums within each epoch, allowing us to absorb the residual bias into weighted gradient second moments without an extra factor of $n$. Together with the gradient-moment self-bound, this closes a harmonic recursion under $\eta L\log(nK)\lesssim1$.}

\subsection{Permutation Stability and Self-Bounding Estimates}
\label{subsec:technical_proof_part_1}
The current iterate in RR-SGD depends on the components already visited in its epoch. Consequently, sampled function values and gradient moments cannot be directly replaced by their full averages. Following the arguments from \cite{liu2026random}, we control this dependence by comparing the original trajectory with a trajectory generated by swapping two positions of the permutation.

Given a permutation $\pi \in S_n$ and two indices $i,j\in [n]$ satisfying the $j\leq i$, Let $\pi(i,j):=\pi\circ (ij)$ denote the permutation obtained by swapping the $i$-th and $j$-th entries of $\pi$.

In words, $\pi(i,j)$ is the permutation generated by exchanging the elements. Equipped with the above notation, we can define the virtual iterate $\rvx_{k}^{i}(i, j)$ as the $i$-th iterate of the $k$-th epoch with the permutation $\pi_{k}(i, j)$, i.e., for all $l\in [i]$
\begin{equation}
    \rvx_{k}^{l+1}(i, j) = \rvx_{k}^{l}(i, j) - \eta_k \nabla f_{\pi_{k}^{l}(i, j)}(\rvx_{k}^{l}(i, j))\label{equ:permutated_update}
\end{equation}
where $\rvx_{k}^{1}(i, j):=\rvx_{k}^{1}$.

The two trajectories coincide before the first swapped position. After that position and before the second swapped update, they apply the same sequence of component-gradient maps. Nonexpansiveness of these maps therefore controls the discrepancy between the two trajectories.

{
Following the above observation and based on Lemma 3 in~\cite{liu2026random}, we have the lemmas below to bound the drift term.} 
\begin{lemma}
    \label{lem:permutation_stability}
    Under {Assumptions~\ref{ass:convexity} and~\ref{ass:smoothness}}, suppose $\eta\leq 2/L$, $\forall k\in [K]$, then for any $k\in [K]$, $i\in [n]$, and $j\in [i-1]$, RR-SGD (Algorithm~\ref{alg:rr_sgd}) satisfies
    \begin{equation}
        \ltwonorm{\rvx_{k}^{i}(i, j) - \rvx_{k}^{i}} \leq \eta \ltwonorm{\nabla f_{\pi_{k}^{i}}(\rvx_{k}^{j}) - \nabla f_{\pi_{k}^{j}}(\rvx_{k}^{j})}\label{equ:4.3}
    \end{equation}
    and
    \begin{equation}
        \E\sqbk{\ltwonorm{\rvx_{k}^{i}(i, j) - \rvx_{k}^{i}}^2} \leq 4\eta^2 \E\sqbk{\ltwonorm{\nabla f_{\pi_{k}^{j}}(\rvx_{k}^{j})}^2}\label{equ:4.4}
    \end{equation}
\end{lemma}

{\Cref{equ:4.3} is the swap-stability bound of Lemma 3 in~\citet{liu2026random}, and~\cref{equ:4.4} gives its second-moment consequence. Crucially, the perturbation is controlled by the earlier update $j$, without an additional factor dependence on the separation $i-j$. We retain this per-update structure rather than collapse it into an epoch-wide bound, allowing the suffix weights to control the history terms later.}

% \textcolor{red}{Equation~\cref{equ:4.3} is the stability estimate in Lemma~3 of \citet{liu2026random}. Equation~\cref{equ:4.4} makes its second-moment consequence explicit. Here we apply this drift control to both the sampling bias $\beta_k^i$ and the sampled-gradient moment $Q_k^i$, obtaining the coupled self-bounding estimates below while retaining the contribution of earlier gradient moments within the epoch. Our last-iterate analysis combines these estimates with the signed bias at random comparison points, so that the suffix weights control repeated contributions without summing over complete epochs.}
\begin{lemma}[Self-bounding estimates]
    \label{lem:self_boundness}
    When $0\leq \eta \leq 1/(2L)$, for any $k\in [K]$ and $i \in [n]$
    \begin{align}
        &\E\sqbk{f(\rvx_{k}^{i}) - f_{\pi_{k}^{i}}(\rvx_{k}^{i})}\leq \frac{\eta}{2} \E\sqbk{\ltwonorm{\nabla f_{\pi_{k}^{i}}(\rvx_{k}^{i})}^2}\nonumber\\
        &\qquad + \frac{3\eta}{n}\sum_{j < i} \E\sqbk{\ltwonorm{\nabla f_{\pi_{k}^{j}}(\rvx_{k}^{j})}^2}.\label{equ:4.10}
    \end{align}
    When $\eta \leq 2/L$, then for any $k\in [K]$ and $i\in[n]$
    \begin{align}
        &\E\sqbk{\ltwonorm{\nabla f_{\pi_{k}^{i}}(\rvx_{k}^{i})}^2} \leq 8L \E\sqbk{f(\rvx_{k}^{i}) - f(\rvx_{*})}\nonumber\\
        &\qquad + 4\sigma_{*}^2 + \frac{8\eta^2 L^2}{n}\sum_{j < i}\E\sqbk{\ltwonorm{\nabla f_{\pi_{k}^{j}}(\rvx_{k}^{j})}^2}
        \label{equ:4.11}
    \end{align}
\end{lemma}
{\Cref{equ:4.10} adapts the sampling bias estimate of Lemma 5 in~\citet{liu2026random}. The additional ingredient is~\cref{equ:4.11}: applying the permutation identity to squared gradient norms relates $Q_{k}^{i}$ to the full component average, while retaining the history term $(8\eta^2L^2/n)\sum_{j<i}Q_{k}^{j}$. The full average is controlled by $\Delta_{k}^{i}$ and $\sigma_{*}^{2}$, without a global gradient bound. Crucially, the $\eta^2 L^2/n$ scaling makes the history term compatible with last-iterate weighting: combining Lemma~\ref{lem:harmonic_weight_block_sum} bounds its total contribution by $16\eta^2 L^2 \cH_{N}W_{N}$ as in~\cref{equ:4.26}. This term is absorbable under the theorem's stepsize condition, closing the moment estimate without an inverse-$n$ stepsize restriction.}

% \textcolor{red}{Equation~\cref{equ:4.10} is a variant of the sampling-bias bound in Lemma~5 of \citet{liu2026random}. It bounds the sampling bias in terms of sampled-gradient second moments, leaving these moments to be controlled.~\cref{equ:4.11} supplies the additional estimate needed for our last-iterate analysis: it bounds the sampled-gradient moment $\|\nabla f_{\pi_{k}^{i}}(\rvx_{k}^{i})\|^2$ by the current objective gap $f(\rvx_{k}^{i}) - f(\rvx_{*})$, the noise at the minimizer $\sigma_{*}^2$, and earlier gradient moments within the epoch.}

% \textcolor{red}{The key observation is that the permutation identity also applies to squared gradient norms. Together with Lemma~\ref{lem:permutation_stability}, this lets us compare the sampled second moment with the full component average despite the dependence between $\pi_k^i$ and $\rvx_k^i$, while retaining a correction involving earlier $\|\nabla f_{\pi_{k}^{i}}(\rvx_{k}^{j})\|^2$ with $j<i$. The coefficient $8\eta^2L^2/n$ of this history term is essential: under the harmonic weights used in Lemma~\ref{lem:4.5}, its accumulated contribution can be absorbed without an extra factor of $n$. Combined with the comparison-point bias cancellation developed next, this control yields a closed recursion in the objective gaps and permits the large stepsizes in Theorem~\ref{thm:main}.}

\subsection{Suffix-to-Last Conversion}
\label{subsec:technical_proof_part_2}
%A comparison with the fixed minimizer is sufficient for an averaged-output estimate, but a suffix-to-last argument also requires comparisons with earlier random iterates. Under RR-SGD, such an iterate depends on the components already revealed in its epoch, so the remaining components are not unbiased at that comparison point. We retain this comparator bias explicitly throughout the conversion.~RR-SGD contains the outer loop indexed by $k$ and inner loop indexed by $i\in [n]$. 
At epoch index $k$ and inner iteration index $i$, the corresponding global iteration index is $\tau(k, i)$.
For a fixed $N\geq 2$, let $(k_N,i_N)$ be the unique pair satisfying
\[
N=\tau(k_N,i_N), \quad 1\leq i_N\leq n.
\]
We write $\Delta_N:=\Delta_{k_N}^{i_N}$ for the objective gap at global position $N$.
Define the index set
\[
\cI_N(p,q)
:=
\set{(k,i): 
\tau(p,q)\leq \tau(k,i)\leq N},
\]
which contains all iteration pairs from $(p,q)$ through global iteration $N$, inclusive. Let $\cI_N:=\cI_N(1,1)$ denote the set of iteration pairs corresponding to the first $N$ steps.
For each $(k, i)\in \cI_{N}$, let
\[
e_{N, k}^{i} := \min\set{nk, N}
\]
denote the global index of the last position of epoch $k$ included in the horizon $N$.
To record the preceding positions within epoch $k$, define
\[
\cP_k^i:=\set{(k,j):1\leq j<i},
\qquad k\geq1,\quad i\in[n].
\]
In particular, $\cP_k^1=\varnothing$.

Fix a suffix starting at $(p,q)$ and use $\rvx_p^q$ as its comparison point. Within epoch $p$, all unrevealed positions have the same conditional marginal distribution. In later epochs, the fresh permutation is independent of $\rvx_p^q$. These observations give an exact expression for the comparator bias and the resulting suffix inequalities.
\begin{lemma}
\label{lem:exact_comparator_bias_identify}
    For every $(p, q)\in \cI_{N}$, and every $(k, i)\in \cI_{N}(p, q)$, we have
    \begin{align}
        &\E\sqbk{f_{\pi_{k}^{i}}(\rvx_{p}^{q}) - f(\rvx_{p}^{q})} = -\beta_{p}^{q}\mathbf{1}_{\set{k=p}};\label{equ:4.16}\\
        &\sum_{(k,i)\in\cI_{N}(p, q)}\!\!\!\!(\Delta_{k}^{i} - \Delta_{p}^{q})\leq -(e_{N, p}^{q} - \tau(p, q) + 1)\beta_{p}^{q} \nonumber\\
        &+\frac{\eta}{2}\!\!\!\!\sum_{(k, i)\in \cI_{N}(p, q)}\!\!\!\E\sqbk{\ltwonorm{\nabla f_{\pi_{k}^{i}}(\rvx_{k}^{i})}^2}+\!\!\!\!\sum_{(k, i)\in \cI_{N}(p, q)}\!\!\!\!\beta_{k}^{i} ;\label{equ:4.17}\\
        & \sum_{(k,i)\in \cI_{N}}\Delta_{k}^{i}\leq \frac{D^2}{2\eta} + \sum_{(k,i)\in\cI_{N}}\beta_{k}^{i}\nonumber\\
        &\qquad\qquad + \frac{\eta}{2}\sum_{(k,i)\in\cI_{N}}\E\sqbk{\ltwonorm{\nabla f_{\pi_{k}^{i}}(\rvx_{k}^{i})}^2} \label{equ:4.18}
    \end{align}
\end{lemma}
We now apply the suffix-to-last identity to these inequalities. {Here, the \emph{comparator bias} refers to $\E[f_{\pi_{k}^{i}}(\rvx_{p}^{q}) - f(\rvx_{p}^{q})]$, the discrepancy between the sampled component loss and the full objective at the random comparison point $\rvx_{p}^{q}$. By~\cref{equ:4.16}, this quantity equals $-\beta_{p}^{q}\mathbf{1}_{\{k=p\}}$. We retain its sign rather than bound it separately. Combining its contribution with the trajectory sampling bias contribution yields the weights below.}
\begin{lemma}
\label{lem:4.5}
    We have
    \begin{equation}
        \label{equ:4.20}
        \begin{aligned}
            &\Delta_{N}\leq \frac{D^2}{2\eta N} + \frac{\eta}{2}\sum_{(k, i)\in \cI_{N}}w_{N, k}^{i}\E\sqbk{\ltwonorm{\nabla f_{\pi_{k}^{i}}(\rvx_{k}^{i})}^2}\\
            &\qquad + \sum_{(k, i)\in \cI_{N}}v_{N, k}^{i}\beta_{k}^{i}            
        \end{aligned}
    \end{equation}
    where $v_{N,k}^{i}$, $w_{N,k}^{i}$ are defined as
    \begin{align*}
        w_{N, k}^{i}\! :=&\frac{1}{\max\set{N-\tau(k, i), 1}}\\
        v_{N,k}^{i}\!:=&\!\begin{cases}
        \dfrac{N - e_{N, k}^{i}}{(N-\tau(k, i))(N- \tau(k, i) + 1)},\!\!\!  &\tau(k, i)<N,\\
        1, \!\!\! &\tau(k, i)=N.
    \end{cases}
    \end{align*}
\end{lemma}
{The weights play different roles: $w_{N,k}^i$ weights the sampled-gradient moments, whereas $v_{N,k}^i$ is the residual coefficient of the signed sampling bias. The cancellation is exact at every nonterminal position in the final truncated epoch, where $v_{N,k}^i=0$. Moreover, the residual weights are nonnegative and satisfy the terminal-inclusive block-sum bound established in Appendix~\ref{app:subsec:proof_of_prop:weighted_gradient_endpoint_inequality}. These properties allow us to apply~\cref{equ:4.10} and absorb its within-epoch history terms into $W_N$, yielding Proposition~\ref{prop:weighted_gradient_endpoint_inequality}.}

% Intuitively, $w_{N, k}^{i}$ is the gradient weights generated by suffix transformation, and $v_{N, k}^{i}$ is the gradient weights obtained after retaining the comparator bias. This nonnegativity permits the use of the one-sided bias bound, while their block-sum control prevents within-epoch repetitions from introducing an additional factor proportional to $n$. 

%The new technique introduced here is the combination of the suffix identify with the deviation from the exact comparison point to obtain the weights $v_{N,k}^{i}$ that satisfy both the non-negativity condition and the in-round summation constraint.

%The non-negativity property permits the use of a unilateral deviation bound; the in-wheel summation bound, in turn, ensures that historical deviations can be converted into existing endpoint gradient weights. Together, these two form the basis for subsequent recursive calculations.

% We can now eliminate the explicit bias terms $\beta_{k}^{i}$. Applying the self-bounding estimate for $\beta_k^i$ in~\cref{equ:4.10} and exchanging the order of summation and combine the weight relation of $v_{N, k}^{i}$ and $w_{N,k}^{i}$ with $v_{N, k}^{i} \leq w_{N, k}^{i}$ (Lemma~\ref{lem:weight_relation}) satisfies the block-sum inequality absorbs the within-epoch history terms into the weighted sampled-gradient moments $W_N$.
\begin{proposition}
\label{prop:weighted_gradient_endpoint_inequality}
    If $\eta L \leq 1/2$, then for any $N\geq 2$
    \begin{equation}
        \Delta_N \leq \frac{D^2}{2\eta N} +4\eta \underbrace{\sum_{(k,i)\in \cI_{N}}w_{N, k}^{i}\E\sqbk{\ltwonorm{\nabla f_{\pi_{k}^{i}}(\rvx_{k}^{i})}^2}}_{:=W_N} \label{equ:4.25}
    \end{equation}
\end{proposition}
Proposition~\ref{prop:weighted_gradient_endpoint_inequality} reduces the
last-iterate analysis to controlling a weighted sum of
sampled-gradient second moments and the detailed proof is provided in the Appendix~\ref{app:subsec:proof_of_prop:weighted_gradient_endpoint_inequality}. In~\cref{equ:4.25}, the sampling bias no longer appears explicitly. 
%We next use the second self-bounding inequality to express these moments in terms of objective gaps and close the resulting recursion.

\subsection{Closing the Recursion: The Large-Step Guarantee}
\label{subsec:technical_proof_part_3}
We now control the weighted gradient moments in Proposition~\ref{prop:weighted_gradient_endpoint_inequality}. Multiplying the sampled-gradient self-bound~\cref{equ:4.11} by $w_{N,k}^i$, summing over the horizon, and applying the block-sum estimate~\cref{equ:4.26} gives
\begin{align}
    &W_N \leq 8L \sum_{(k,i)\in\cI_N} w_{N, k}^{i}\Delta_{k}^{i}  + 4\sigma_{*}^2 \sum_{(k,i)\in\cI_N}w_{N, k}^{i}\nonumber\\
    & +8 \eta^2 L^2 \sum_{(p, q)\in \cI_{N}}\E\sqbk{\ltwonorm{\nabla f_{\pi_{p}^{q}}(\rvx_{p}^{q})}^2}{\bigg(\frac{1}{n}\sum_{\substack{(k,i)\in\cI_N\\ (p,q)\in\cP_k^i}}w_{N,k}^{i}\bigg)}\label{equ:second_moment_weight}
\end{align}
The weight of $\E[\|\nabla f_{\pi_{k}^{i}}(\rvx_{k}^{i})\|^2]$ in~\cref{equ:second_moment_weight} can be controlled by the harmonic weights:
\begin{lemma}
\label{lem:harmonic_weight_block_sum}
    For every $(p, q)\in \cI_{N}$, let $\cH_{N}:=\sum_{l=1}^{N}\frac{1}{l}$ be the $N$-th harmonic number, then we have
    \begin{equation}
        \frac{1}{n}\sum_{(k, i)\in \cI_{N}, (p, q)\in \cP_{k}^{i}} w_{N, k}^{i} \leq 2\cH_{N}w_{N, p}^{q}.\label{equ:4.26}
    \end{equation}
\end{lemma}
Plug~\cref{equ:4.26} into~\cref{equ:second_moment_weight}, we have
\begin{equation}
    \label{equ:4.27}
    \begin{aligned}
        &W_{N}\leq 8L\sum_{(k,i)\in\cI_N} w_{N, k}^{i}\Delta_{k}^{i} + 4\sigma_{*}^{2} (\cH_{N-1} + 1)\\
        &\qquad + 16\eta^{2}L^2\cH_{N}W_N        
    \end{aligned}
\end{equation}
Under $\eta L \cH_{N}\leq 1/512$, the last term can be absorbed into the left-hand side, yielding
\begin{equation}
    W_N\leq 16L \sum_{(k,i)\in\cI_N} w_{N, k}^{i}\Delta_{k}^{i} +8\sigma_{*}^{2}(\cH_{N-1} + 1)\label{equ:4.29}
\end{equation}
Substituting~\cref{equ:4.29} into~\cref{equ:4.25}, we obtain
\begin{equation}
    \Delta_{N}\leq \frac{D^2}{2\eta N} + 64\eta L \sum_{(k,i)\in\cI_N} w_{N, k}^{i}\Delta_{k}^{i} + 32\eta \sigma_{*}^{2}(\cH_{N-1} + 1)\label{equ:4.30}
\end{equation}

% The following block-sum estimate controls their total contribution under the harmonic weights.
% \begin{lemma}[Harmonic-weight block sum]
% \label{lem:harmonic_weight_block_sum}
%     For every $(p, q)\in \cI_{N}$, we have 
%     \begin{equation}
%         \frac{1}{n}\sum_{\substack{(k, i)\in \cI_{N}\\ (p, q)\in \cP_{k}^{i}}} w_{N, k}^{i} \leq 2\cH_{N}w_{N, p}^{q}.\label{equ:4.26}
%     \end{equation}
% \end{lemma}

Thus, the weighted second moments $W_N$ of the sampled gradients are controlled by the weighted error sequence and noise term. The remaining step is to close the recursion in the objective gaps. Since the terminal weight equals one,
\[
\sum_{(k,i)\in\cI_N} w_{N, k}^{i}\Delta_{k}^{i}=\Delta_N+\sum_{(k, i)\in \cI_{N}, \tau(k, i)<N}\frac{\Delta_{k}^{i}}{N-\tau(k, i)}.
\]

The identity $\sum_{t=1}^{N-1}[t(N-t)]^{-1}=2H_{N-1}/N$ shows that suffix weighting preserves the $1/N$ decay of the initial-error term up to a harmonic factor. After isolating the terminal gap in~\cref{equ:4.30}, the induction closes under $\eta LH_N\lesssim1$ rather than $\eta nL\lesssim1$. We then have the lemma below.

\begin{lemma}[All-time last-iterate bound]
    \label{lem:all_time_last_iterate_bound}
    For a fixed $N\geq 2$, let {$N=\tau(k_N,i_N)$ and $\Delta_N=\Delta_{k_N}^{i_N}$}. If $\eta L \cH_{N}\leq 1/512$, then
    \[
    \Delta_N \leq \frac{D^2}{\eta N} + 128 \eta\sigma_{*}^{2}\cH_{N}.
    \]
\end{lemma}
Reindex the horizon $N$ by the epoch and inner-iteration indices $(k,i)$, and recall $\tau(k,i)$ denotes the corresponding global position. Since $\Lambda = 1 + \log (nK)$ and $\cH_{nK + 1}\leq 2\Lambda$, the stepsize condition $\eta \leq 1/(1024 L \Lambda)$ gives
{\begin{align}
&\E\sqbk{f(\rvx_k^i)-f(\rvx_*)}\lesssim\frac{D^2}{\eta\tau(k,i)}+\Lambda\eta\sigma_*^2,
\label{equ:all_time_large_step}
\end{align}
for every fixed $(k,i)$ with $i\in[n]$ and $2\leq\tau(k,i)\leq nK+1$. At global position $nK+1$, we use the representation $(k,i)=(K+1,1)$. The trajectory may be extended by a fresh auxiliary permutation for this next epoch; the comparison uses its first update only in the analysis and requires no additional gradient evaluation by the algorithm.}

~\cref{equ:all_time_large_step} controls every component-time iterate with an SGD-scale noise term throughout the admissible stepsize range. For smaller stepsizes, however, a reshuffling-specific quadratic noise term can be sharper. We combine the large-step estimate with the following existing small-step endpoint bound in Lemma~\ref{lem:theorem_c_1} and a within-epoch comparison in Lemma~\ref{lem:within_epoch_comparison} to recover that dependence in Section~\ref{subsec:putting_together}.

% The preceding lemma controls every component-time iterate. To obtain the returned endpoint after $T=nK$ updates, we apply it at global position $N=T+1$, corresponding to $\rvx_{K+1}^1$. \textcolor{blue}{This gives the following special case.}
% \begin{theorem}[Large-step endpoint bound]
%     \textcolor{blue}{Let $\Lambda=1+\log(nK)$ and $\eta_k\equiv\eta$. When $0<\eta \leq \tfrac{1}{1024 L \Lambda }$,}
%     \[
%     \E\sqbk{f(\rvx_{K+1}^{1}) - f(\rvx_{*})} \lesssim \textcolor{blue}{\frac{D^2}{\eta nK}} + 256\Lambda \eta \sigma_{*}^{2}.
%     \]
% \end{theorem}
% This establishes the new large-step estimate. It controls \textcolor{blue}{each fixed component-time iterate} with an SGD-scale noise term throughout the admissible stepsize range. For smaller stepsizes, however, a reshuffling-specific quadratic noise term can be sharper. We recover that dependence by combining the large-step estimate with \textcolor{blue}{an existing small-step endpoint bound and a within-epoch comparison}.
\subsection{Putting Together}
\label{subsec:putting_together}
The large-step estimate provides the noise term
$\eta\sigma_\star^2$ up to logarithmic factors.
To obtain our results in both small- and large-step regimes, we combine the following
last-iterate guarantee in previous work \citep{liu2024last}.
{
\begin{lemma}[Theorem C.1 in~\citet{liu2024last}]
    \label{lem:theorem_c_1}
    Under Assumptions~\ref{ass:minimizer}, \ref{ass:convexity}, and~\ref{ass:smoothness}, for any $m\geq2$, if $0<\eta\leq1/(4nL\sqrt{2(1+\log m)})$, the endpoint $\rvx_{m+1}^1$ after $m$ epochs of RR-SGD with constant stepsize $\eta$ satisfies
    \[
    \E\sqbk{f(\rvx_{m+1}^{1}) - f(\rvx_{*})}
    \lesssim\frac{D^2}{\eta nm}+\eta^2nL\sigma_*^2 \log m.
    \]
\end{lemma}
The following comparison extends this endpoint guarantee to the iterates within the next epoch.
\begin{lemma}[Within-epoch comparison]
    \label{lem:within_epoch_comparison}
    Under Assumptions~\ref{ass:minimizer}, \ref{ass:convexity}, and~\ref{ass:smoothness}, suppose each epoch uses a fresh independent uniform permutation and $0<\eta nL\leq1/2$. Then, for every epoch $k$ and every $1\leq i\leq n+1$,
    \[
    \Delta_k^i\leq6\Delta_k^1+8\eta^2nL\sigma_*^2.
    \]
\end{lemma}
The proof is given in Appendix~\ref{subsec:app:within_epoch_comparison}. It controls the squared displacement from the epoch's starting point using the second moment of a uniformly sampled prefix, then applies the smoothness of $f$.

Fix $(k,i)$ with $i\in[n]$ and $2\leq\tau(k,i)\leq nK+1$. We set threshold
\[
\eta_0:=\frac{1}{4nL\sqrt{2(1+\log K)}}.
\]

\paragraph{Small stepsizes (Lemma~\ref{lem:theorem_c_1}): $\eta\le\eta_0$.}
For $k\geq3$, apply Lemma~\ref{lem:theorem_c_1} with $m=k-1\leq K$ to the epoch's starting point $\rvx_k^1$, then apply Lemma~\ref{lem:within_epoch_comparison}. Since $\tau(k, i)\leq2n(k-1)$ and $1+\log(k-1)\leq\Lambda$, this gives
\[
\Delta_k^i\lesssim\frac{D^2}{\eta \tau(k, i)}+\Lambda\eta^2nL\sigma_*^2.
\]
For the first two epochs where $k\leq 2$, applying Lemma~\ref{lem:within_epoch_comparison} once or twice gives the same bound using $\Delta_1^1\leq LD^2/2$, $N\leq2n$, and $\eta nL\leq1/2$. Thus the bound covers every admissible $(k,i)$. Since $\eta nL<1$ and $\Lambda\geq1$, it implies the claimed $\Lambda^{3/2}\min\{\eta\sigma_*^2,\eta^2nL\sigma_*^2\}$ term.

\paragraph{Intermediate stepsizes (Lemma~\ref{lem:all_time_last_iterate_bound}):
$\eta>\eta_0$ and $\eta nL<1$.}
In this regime,
\[
\frac{1}{\eta nL}
<
4\sqrt{2(1+\log K)}
\leq4\sqrt{2\Lambda}.
\]
consequently, the noise term in the all-time large-step estimate~\cref{equ:all_time_large_step}
satisfies
\[
\Lambda\eta\sigma_\star^2
=
\frac{\Lambda}{\eta nL}
\eta^2nL\sigma_\star^2
\lesssim
\Lambda^{3/2}\eta^2nL\sigma_\star^2.
\]
Thus, the large-step estimate also yields the desired
reshuffling noise scale in the intermediate regime,
with an additional logarithmic factor.

\paragraph{Larger stepsizes (Lemma~\ref{lem:all_time_last_iterate_bound}): $\eta nL\ge1$.}
Here,
\[
\min\{\eta\sigma_\star^2,
       \eta^2nL\sigma_\star^2\}
=
\eta\sigma_\star^2,
\]
so~\cref{equ:all_time_large_step} gives the required bound directly.

Combining the three regimes proves Theorem~\ref{thm:main} for every fixed component-time iterate in the stated horizon. Taking $(k,i)=(K+1,1)$ recovers the returned endpoint. When $\sigma_*^2=0$, the largest admissible stepsize gives $\widetilde O(LD^2/\tau(k,i))$, and hence $\widetilde O(LD^2/(nK))$ at the endpoint.
}

\section{Discussion}\label{sec:discussion}
For convex, smooth components, we establish expected last-iterate guarantees for RR-SGD with constant component stepsizes up to order $1/(L\log(nK))$, removing the explicit $1/n$ factor from the stepsize restriction. The result applies at every prescribed component time, and the tuned endpoint rate matches the averaged-output rate up to logarithmic factors. Our analysis retains the signed bias at random comparison points and exploits its partial cancellation with the trajectory bias. The resulting weights, together with a sampled-gradient second-moment bound, allow the harmonic recursion to close without an extra factor of $n$.

An immediate question for future work is whether the logarithmic losses can be removed, permitting component stepsizes of order $1/L$. Non-constant stepsizes may also improve last-iterate convergence rates by using larger early steps to reduce the initial error and smaller later steps to control the noise contribution. Finally, retaining the individual smoothness constants $L_i$ in the drift estimates may clarify which error terms can depend on an average of these constants rather than their maximum.

\bibliography{reference}

\clearpage
\appendix
\appendixpage

\section{Notation and Technical Lemmas}
\subsection{Notation Tables}
For readability, we recall and summarize the notation used in the paper and define the auxiliary quantities used in the proof.
\begin{itemize}
    \item $f_{i}$: the $i$-th component function of the finite-sum optimization problem, $f_{i}$ is convex and differentiable.
    \item $f$: the optimization objective, $f(\rvx) = \frac{1}{n}\sum_{i=1}^{n}f_{i}(\rvx)$.
    \item $D$: the distance between the initial point and the minimizer, $D := \ltwonorm{\rvx_{1}^{1}-\rvx_{*}}$.
    \item $\rvx_k^i$: the iteration point at position $i$ of epoch $k$.
    \item $\eta$: the stepsize.
    \item $\sigma_{*}^2$: the variance of the gradients at the minimizer, $\sigma_{*}^{2} := \frac{1}{n} \sum_{i=1}^{n} \|\nabla f_{i}(\rvx_{*})\|^2$.
    \item $k,i$: the epoch index and the within-epoch position, respectively.
    \item $\tau(k,i)$: the global iteration index, $\tau(k,i):=(k-1)n+i$.
    \item $\pi_k^{i}$: the component index sampled at position $i$ of epoch $k$.
    \item $\Delta_k^i$: the expected suboptimality at position $i$ of epoch $k$, 
    \[\Delta_k^i:=\E\sqbk{f(\rvx_k^i)-f(\rvx_{*})}\]
    \item $Q_{k}^{i}$: the expected squared norm of the gradient at position $i$ of epoch $k$, 
    \[Q_{k}^{i}:=\E\sqbk{\ltwonorm{\nabla f_{\pi_k^{i}}(\rvx_k^i)}^2}\]
    \item $\beta_{k}^{i}$: the expected bias of the component function at position $i$ of epoch $k$, 
    \[
    \beta_{k}^{i}:=\E\sqbk{f(\rvx_k^i)-f_{\pi_k(i)}(\rvx_k^i)}
    \]
    \item $M_{k}^{i}$: the expected average squared norm of the gradients at position $i$ of epoch $k$, 
    \[
    M_k^i:=\E\sqbk{\frac{1}{n}\sum_{r=1}^{n}\ltwonorm{\nabla f_r(\rvx_k^i)}^2}
    \]
    \item $W_N$: the weighted sum of the expected squared norm of the gradients up to iteration $N$,
    \[
    W_N := \sum_{(k,i)\in\cI_N}w_{N,k}^{i}Q_k^i
    \]
    \item $V_N$: the weighted sum of the expected suboptimality up to iteration $N$,
    \[
    V_N := \sum_{(k,i)\in\cI_N}w_{N,k}^{i}\Delta_k^i
    \]
    \item $\cH_{N}$: the $N$-th harmonic number, $\cH_{N}:=\sum_{l=1}^{N}\frac{1}{l}$.
\end{itemize}

% We list the notation in Table~\ref{tab:parameters_ofMDP}.
% \begin{table}[ht]
%     \centering
%     \caption{Notation for RR-SGD}
%     \begin{tabular}{|>{\centering\arraybackslash}m{0.23\textwidth}|>{\centering\arraybackslash}m{0.67\textwidth}|}
%     \hline
%     $f_i$ & $f_i$ is convex and differentiable \\
%     $f$ & $f$ is the optimization objective \\
%     $D$ & $D := \ltwonorm{\rvx_{1}^{1}-\rvx_{*}}$ \\
%     $\rvx_k^i$ & iteration point at position $i$ of epoch $k$ \\
%     $\eta$ & stepsize \\
%     $\sigma_{*}^2$ & $\sigma_{*}^{2} := \frac{1}{n} \sum_{i=1}^{n} \|\nabla f_{i}(\rvx_{*})\|^2$ \\ 
%     $k,i$ & epoch index and within-epoch position, respectively \\
%     $\tau(k,i)$ & $\tau(k,i):=(k-1)n+i$, the global iteration index \\
%     $\pi_k(i)$ & component index sampled at position $i$ of epoch $k$ \\
%     $\Delta_k^i$ & $\Delta_k^i:=\E\sqbk{f(\rvx_k^i)-f(\rvx_{*})}$ \\
%     $Q_{k}^{i}$ & $Q_{k}^{i}:=\E\sqbk{\ltwonorm{\nabla f_{\pi_k(i)}(\rvx_k^i)}^2}$\\
%     $\beta_{k}^{i}$ & $\beta_{k}^{i}:=\E\sqbk{f(\rvx_k^i)-f_{\pi_k(i)}(\rvx_k^i)}$\\
%     $M_{k}^{i}$ & $M_k^i:=\E\sqbk{\frac{1}{n}\sum_{r=1}^{n}\ltwonorm{\nabla f_r(\rvx_k^i)}^2}$\\
% \hline
%     \end{tabular}
%     \label{tab:parameters_ofMDP}
% \end{table}

For $j<i$, the swapped permutation $\pi(i,j)$ is defined componentwise by
\[
\pi^{r}(i,j)=
\begin{cases}
\pi^i, & r=j,\\
\pi^j, & r=i,\\
\pi^r, & r\notin\{i,j\},
\end{cases}
\qquad r\in[n].
\]
For $j=i$, $\pi(i,i)=\pi$.

\subsection{Technical Lemmas}
The proofs of our main Theorem and Lemmas in the main text involve several technical lemmas, which are summarized below.
\begin{lemma}[Co-coercivity, Lemma 1 in~\citet{liu2026random}]
    \label{lem:co_coercivity}
    Let $h:\R^{d}\to \R$ be a differentiable convex function that is also $L$-smooth, then we have, for any $\rvx, \rvy\in \R^{d}$,
    \[
    \ltwonorm{\nabla h(\rvx) - \nabla h(\rvy)}^2 \leq 2L \bracket{h(\rvx) - h(\rvy) - \langle \nabla h(\rvy), \rvx - \rvy \rangle}
    \]
\end{lemma}

\begin{lemma}[Lemma 2 in~\citet{liu2026random}]
    \label{lem:permutation_inequality}
    Given an arbitrary finite-sum function $\ell(\rvx) = \frac{1}{n}\sum_{i=1}^{n} \ell_{i}(\rvx)$, suppose RR-SGD is employed, then for all $k\in [K]$ and $i\in [n]$, we have
    \[
    \E\sqbk{\ell(\rvx_{k}^{i}) - \ell_{\pi_{k}^{i}}(\rvx_{k}^{i})} = \frac{1}{n} \sum_{j<i}\E\sqbk{\ell_{\pi_{k}^{i}}(\rvx_{k}^{i}(i, j)) - \ell_{\pi_{k}^{i}}(\rvx_{k}^{i})}
    \]
    where $\rvx_{k}^{i}(i, j)$ is defined in~\cref{equ:permutated_update}.
\end{lemma}

\begin{lemma}[Non-expansiveness and the best-point self-boundary]
    \label{lem:non_expansiveness}
    For any convex, differentiable, and $L$-smooth function $f:\R^d\to\R$, any $\rvx,\rvy\in\R^d$, and $0\leq\eta\leq2/L$,
    \begin{equation}
        \ltwonorm{(\rvx - \eta \nabla f(\rvx)) - (\rvy - \eta \nabla f(\rvy))}\leq \ltwonorm{\rvx - \rvy}\label{equ:4.1}
    \end{equation}
    Moreover, under Assumptions~\ref{ass:minimizer},~\ref{ass:convexity}, and~\ref{ass:smoothness}, the finite-sum objective $f=\frac1n\sum_{i=1}^n f_i$ satisfies, for every $\rvx\in\R^d$,
    \begin{equation}
        \frac{1}{n}\sum_{i=1}^{n}\ltwonorm{\nabla f_{i}(\rvx)}^2 \leq 4 L \bracket{f(\rvx) - f(\rvx_{*})} + 2\sigma_{*}^2\label{equ:4.2}
    \end{equation}
\end{lemma}
\begin{proof}
    First, for a fixed $\rvy$, define
\[
h_{\rvy}(\rvx) := f(\rvx) - f(\rvy) - \langle \nabla f(\rvy), \rvx - \rvy\rangle 
\]
By convexity, $h_{\rvy}(\rvx)\geq 0$, and $h_{\rvy}$ is $L$-smooth. we have
\[
0\leq h_{\rvy}\bracket{\rvx - \frac{1}{L}\nabla h_{\rvy}(\rvx)} \leq h_{\rvy}(\rvx) - \frac{1}{2L}\ltwonorm{\nabla h_{\rvy}(\rvx)}^2
\]
Thus,
\[
\ltwonorm{\nabla f(\rvx) - \nabla f(\rvy)}^2 \leq 2L \bracket{f(\rvx) - f(\rvy) - \langle \nabla f(\rvy), \rvx - \rvy \rangle}
\]
switch $\rvx$ and $\rvy$, we have
\[
\ltwonorm{\nabla f(\rvy) - \nabla f(\rvx)}^2 \leq 2L \bracket{f(\rvy) - f(\rvx) - \langle \nabla f(\rvx), \rvy - \rvx \rangle}
\]
Summing the two inequalities, we have
\begin{equation*}
    \frac{1}{L}\ltwonorm{\nabla f(\rvx) - \nabla f(\rvy)}^2 \leq \langle \nabla f(\rvx) - \nabla f(\rvy), \rvx - \rvy \rangle
\end{equation*}
Finally,
\begin{align*}
    \ltwonorm{(\rvx - \eta \nabla f(\rvx)) - (\rvy - \eta \nabla f(\rvy))}^2 =& \ltwonorm{\rvx - \rvy}^2 - 2\eta \langle \nabla f(\rvx) - \nabla f(\rvy), \rvx - \rvy\rangle\\
    &\quad + \eta^2\ltwonorm{\nabla f(\rvx) - \nabla f(\rvy)}^2\\
    \leq& \ltwonorm{\rvx - \rvy}^2 -  \eta \bracket{\frac{2}{L} - \eta}\ltwonorm{\nabla f(\rvx) - \nabla f(\rvy)}^2\\
    \leq& \ltwonorm{\rvx - \rvy}^2
\end{align*}
where the last inequality holds when $0\leq \eta \leq 2/L$, then we complete the proof of~\cref{equ:4.1}. 

\paragraph{Proof of~\cref{equ:4.2}.} Let
\[
h_{i}(\rvx) = f_{i}(\rvx) - f_{i}(\rvx_{*}) - \langle\nabla f_{i}(\rvx_{*}), \rvx - \rvx_{*}\rangle
\]
By the squared triangle inequality and co-coercivity of each component, we have
\begin{align*}
    \frac{1}{n}\sum_{i=1}^{n}\ltwonorm{\nabla f_{i}(\rvx)}^2 \leq& \frac{2}{n}\sum_{i=1}^{n}\ltwonorm{\nabla f_{i}(\rvx) - \nabla f_{i}(\rvx_{*})}^2 + \frac{2}{n}\sum_{i=1}^{n}\ltwonorm{\nabla f_{i}(\rvx_{*})}^2\\
    \leq& \frac{4L}{n}\sum_{i=1}^{n}\bracket{f_{i}(\rvx) - f_{i}(\rvx_{*}) - \langle\nabla f_{i}(\rvx_{*}), \rvx - \rvx_{*}\rangle} + 2\sigma_{*}^2\\
    =& 4L\bracket{f(\rvx) - f(\rvx_{*})} + 2\sigma_{*}^2.
\end{align*}
The last equality uses $\frac1n\sum_{i=1}^n\nabla f_i(\rvx_*)=\nabla f(\rvx_*)=\mathbf0$, since $\rvx_*$ minimizes $f$. This proves~\cref{equ:4.2}.
\end{proof}

% !TeX root = ../arxiv_version.tex
\section{Missing Proofs}

\subsection{Proof of Lemma~\ref{lem:permutation_stability}}
The proof of~\cref{equ:4.3} is exactly the same as the proof of Lemma 3 in~\citet{liu2026random}. We provide the proof of~\cref{equ:4.4} here. The original and swapped trajectories use the same component indices before position $j$, so
\[
\rvx_{k}^{j}(i, j) = \rvx_{k}^{j}
\]
The trajectories may differ after the $j$-th update. By~\cref{equ:permutated_update}, their difference is
\[
\rvx_{k}^{j+1}(i, j) - \rvx_{k}^{j + 1} = -\eta \sqbk{\nabla f_{\pi_{k}^{i}}(\rvx_{k}^{j}) - \nabla f_{\pi_{k}^{j}}(\rvx_{k}^{j})}
\]
For $j+1\leq l\leq i-1$, the two trajectories apply the same component-gradient map. Its non-expansiveness gives
\[
\ltwonorm{\rvx_k^{l+1}(i,j)-\rvx_k^{l+1}}
\leq\ltwonorm{\rvx_k^l(i,j)-\rvx_k^l}.
\]
Iterating this inequality up to position $i$, before the second swapped update, yields
\begin{align*}
    \ltwonorm{\rvx_k^i(i,j)-\rvx_k^i}
    \leq& \ltwonorm{\rvx_{k}^{j+1}(i, j) - \rvx_{k}^{j + 1}}\\
    =& \eta \ltwonorm{\nabla f_{\pi_{k}^{i}}(\rvx_{k}^{j}) - \nabla f_{\pi_{k}^{j}}(\rvx_{k}^{j})}
\end{align*}
When $i=j+1$, the same bound follows directly from the difference after the $j$-th update, with no intervening updates.
Suppose at $k$-th epoch and $j$-th inner iteration, the filtration $\cF_{k}^{j}$:
\[
\cF_{k}^{j} := \sigma\bracket{\pi_{1}, \ldots, \pi_{k-1}, \pi_{k}^{1}, \ldots, \pi_{k}^{j-1}}.
\]
For a fixed $\cF_{k}^{j}$, the iteration point $\rvx_{k}^{j}$ is fixed, and $\pi_{k}^{j}, \pi_{k}^{i}\in [n]\backslash \{\pi_{k}^{1}, \ldots, \pi_{k}^{j-1}\}$ admits the uniform distribution. Thus, we have
\[
\E\sqbk{\ltwonorm{\nabla f_{\pi_{k}^{i}}(\rvx_{k}^{j})}^2\mid \cF_{k}^{j}} = \frac{1}{n-j + 1}\sum_{s\in [n]\backslash \{\pi_{k}^{1}, \ldots, \pi_{k}^{j-1}\}} \ltwonorm{\nabla f_{s}(\rvx_{k}^{j})}^2 = \E\sqbk{\ltwonorm{\nabla f_{\pi_{k}^{j}}(\rvx_{k}^{j})}^{2} | \cF_{k}^{j}}
\]
Finally, taking the total expectation, we have
\[
\E\sqbk{\ltwonorm{\nabla f_{\pi_{k}^{i}}(\rvx_{k}^{j})}^2} = \E\sqbk{\ltwonorm{\nabla f_{\pi_{k}^{j}}(\rvx_{k}^{j})}^2}= Q_{k}^{j}
\]
Here we only relies on the uniform distribution of the permutation $\pi_{k}$, without requiring the independence of $\pi_{k}^{i}$ and $\pi_{k}^{j}$. Thus, we have
\begin{align*}
    \E\sqbk{\ltwonorm{\rvx_{k}^{i}(i, j) - \rvx_{k}^{i}}^{2}} \leq& \eta^2 \E\sqbk{\ltwonorm{\nabla f_{\pi_{k}^{i}}(\rvx_{k}^{j}) - \nabla f_{\pi_{k}^{j}}(\rvx_{k}^{j})}^2}\\
    \leq& 2\eta^2 \E\sqbk{\ltwonorm{\nabla f_{\pi_{k}^{i}}(\rvx_{k}^{j})}^2 + \ltwonorm{\nabla f_{\pi_{k}^{j}}(\rvx_{k}^{j})}^2}\\
    =& 4\eta^2 Q_{k}^{j}
\end{align*}
which completes the proof of~\cref{equ:4.4}.

\subsection{Proof of Lemma~\ref{lem:self_boundness}}
\paragraph{Proof of~\cref{equ:4.10}.} If $\eta=0$, all iterates equal the deterministic initialization $\rvx_1^1$. Since each $\pi_k^i$ is marginally uniform on $[n]$, we have $\beta_k^i=0$, and~\cref{equ:4.10} holds immediately. For the remainder of this argument, assume $\eta>0$. Denote $\delta := \rvx_{k}^{i}(i, j) - \rvx_{k}^{i}$. By smoothness and Young's inequality, we have
\begin{align*}
    f_{\pi_{k}^{i}}(\rvx_{k}^{i} + \delta) - f_{\pi_{k}^{i}}(\rvx_{k}^{i})\leq& \langle \nabla f_{\pi_{k}^{i}}(\rvx_{k}^{i}), \delta \rangle + \frac{L}{2}\ltwonorm{\delta}^2\\
    \leq& \frac{\eta}{2}\ltwonorm{\nabla f_{\pi_{k}^{i}}(\rvx_{k}^{i})}^2 + \frac{1}{2\eta}\ltwonorm{\delta}^2 + \frac{L}{2}\ltwonorm{\delta}^2 \\
    =& \frac{\eta}{2} \ltwonorm{\nabla f_{\pi_{k}^{i}}(\rvx_{k}^{i})}^2 + \bracket{\frac{1}{2\eta} + \frac{L}{2}}\ltwonorm{\delta}^2
\end{align*}
Taking expection on the both sides, combine~\cref{equ:4.4} in Lemma~\ref{lem:permutation_stability}, we have
\[
\E\sqbk{f_{\pi_{k}^{i}}(\rvx_{k}^{i}(i, j)) - f_{\pi_{k}^{i}}(\rvx_{k}^{i})} \leq \frac{\eta}{2}Q_{k}^{i} + \bracket{\frac{1}{2\eta} + \frac{L}{2}}\E\sqbk{\ltwonorm{\delta}^2}\leq\frac{\eta}{2}Q_{k}^{i} + (2\eta + 2\eta^2 L)\E\sqbk{\ltwonorm{\nabla f_{\pi_k^{j}}(\rvx_k^j)}^2}
\]
Apply Lemma~\ref{lem:permutation_inequality} with $l_i = f_{i}$
\begin{align*}
    \E\sqbk{f(\rvx_{k}^{i}) - f_{\pi_{k}^{i}}(\rvx_{k}^{i})} =& \frac{1}{n}\sum_{j<i}\E\sqbk{f_{\pi_{k}^{i}}(\rvx_{k}^{i}(i, j)) - f_{\pi_{k}^{i}}(\rvx_{k}^{i})}\\
    \leq& \frac{1}{n}\sum_{j<i}\sqbk{\frac{\eta}{2}Q_{k}^{i} + (2\eta + 2\eta^2L)\E\sqbk{\ltwonorm{\nabla f_{\pi_k^{j}}(\rvx_k^j)}^2}}\\
    \leq& \frac{\eta}{2}Q_{k}^{i} + \frac{3\eta}{n} \sum_{j<i}\E\sqbk{\ltwonorm{\nabla f_{\pi_k^{j}}(\rvx_k^j)}^2}
\end{align*}
when $\eta L \leq 1/2$, then the coefficient of $\E\sqbk{\ltwonorm{\nabla f_{\pi_k^{j}}(\rvx_k^j)}^2}$ satisfies $2\eta + 2\eta^2 L = 2\eta(1 + \eta L)\leq 3\eta$ when $i\leq n$. Then we complete the proof of~\cref{equ:4.10}

\paragraph{Proof of~\cref{equ:4.11}.} Apply Lemma~\ref{lem:permutation_inequality}, with 
\[
\ell_{j}(\rvx) = \ltwonorm{\nabla f_{j}(\rvx)}^2,\qquad \ell(\rvx):=\frac{1}{n}\sum_{j=1}^{n}\ell_{j}(\rvx)
\]
then
\begin{align*}
    \E\sqbk{\frac{1}{n}\sum_{r=1}^{n}\ltwonorm{\nabla f_r(\rvx_k^i)}^2} - \E\sqbk{\ltwonorm{\nabla f_{\pi_k^{i}}(\rvx_k^i)}^2} =& \frac{1}{n}\sum_{j<i}\bracket{\E\sqbk{\ltwonorm{\nabla f_{\pi_{k}^{i}}(\rvx_{k}^{i}(i, j))}^2 - \ltwonorm{\nabla f_{\pi_{k}^{i}}(\rvx_{k}^{i})}^2}}\\
    =&\frac{1}{n}\sum_{j<i} \E\sqbk{\ltwonorm{\nabla f_{\pi_{k}^{i}}(\rvx_{k}^{i}(i, j))}^2} - \frac{i-1}{n}\E\sqbk{\ltwonorm{\nabla f_{\pi_k^{i}}(\rvx_k^i)}^2}
\end{align*}
\[
\begin{aligned}
M_{k}^{i} - Q_{k}^{i} &= \frac{1}{n}\sum_{j<i}\bracket{\E\sqbk{\ltwonorm{\nabla f_{\pi_{k}^{i}}(\rvx_{k}^{i}(i, j))}^2 - Q_{k}^{i}}}\\
&=\frac{1}{n}\sum_{j<i} \E\sqbk{\ltwonorm{\nabla f_{\pi_{k}^{i}}(\rvx_{k}^{i}(i, j))}^2} - \frac{i-1}{n}\E\sqbk{\ltwonorm{\nabla f_{\pi_k^{i}}(\rvx_k^i)}^2}
\end{aligned}
\]
thus
\begin{equation}
    \label{equ:intermediate_1}
    \sum_{r=1}^{n}\E\sqbk{\ltwonorm{\nabla f_r(\rvx_k^i)}^2} = \sum_{j<i}\E\sqbk{\ltwonorm{\nabla f_{\pi_{k}^{i}}(\rvx_{k}^{i}(i, j))}^2} + (n-(i-1))Q_{k}^{i}
\end{equation}
At the same time,
\begin{align*}
    \ltwonorm{\nabla f_{\pi_{k}^{i}}(\rvx_{k}^{i})}^2 \leq& 2\ltwonorm{\nabla f_{\pi_{k}^{i}}(\rvx_{k}^{i}(i, j))}^2 + 2\ltwonorm{\nabla f_{\pi_{k}^{i}}(\rvx_{k}^{i}) - \nabla f_{\pi_{k}^{i}}(\rvx_{k}^{i}(i, j))}^2\\
    \leq& 2\ltwonorm{\nabla f_{\pi_{k}^{i}}(\rvx_{k}^{i}(i, j))}^2 + 2L^2\ltwonorm{\rvx_{k}^{i} - \rvx_{k}^{i}(i, j)}^2
\end{align*}
Taking the expectation on both sides, we have
\[
\E\sqbk{\ltwonorm{\nabla f_{\pi_k^{i}}(\rvx_k^i)}^2}\leq 2\E\sqbk{\ltwonorm{\nabla f_{\pi_{k}^{i}}(\rvx_{k}^{i}(i, j))}^2} + 2L^2 \E\sqbk{\ltwonorm{\rvx_{k}^{i}(i, j) - \rvx_{k}^{i}}^2}
\]
Taking summation on the both sides for $j<i$, we have
\[
(i-1)\E\sqbk{\ltwonorm{\nabla f_{\pi_k^{i}}(\rvx_k^i)}^2} \leq 2\sum_{j<i}\E\sqbk{\ltwonorm{\nabla f_{\pi_{k}^{i}}(\rvx_{k}^{i}(i, j))}^2} + 2L^2 \sum_{j<i}\E\sqbk{\ltwonorm{\rvx_{k}^{i}(i, j) - \rvx_{k}^{i}}^2}
\]
Combining~\cref{equ:intermediate_1}, we further have
\begin{align*}
    n\E\sqbk{\ltwonorm{\nabla f_{\pi_k^{i}}(\rvx_k^i)}^2} \leq& 2\sum_{j<i}\E\sqbk{\ltwonorm{\nabla f_{\pi_{k}^{i}}(\rvx_{k}^{i}(i, j))}^2} + 2L^2 \sum_{j<i}\E\sqbk{\ltwonorm{\rvx_{k}^{i}(i, j) - \rvx_{k}^{i}}^2}\\
    &\qquad + (n-(i-1))\E\sqbk{\ltwonorm{\nabla f_{\pi_k^{i}}(\rvx_k^i)}^2}\\
    =& 2 \bracket{\sum_{r=1}^{n}\E\sqbk{\ltwonorm{\nabla f_r(\rvx_k^i)}^2} - (n-(i-1))\E\sqbk{\ltwonorm{\nabla f_{\pi_k^{i}}(\rvx_k^i)}^2}} \\
    &\qquad+ 2L^2 \sum_{j<i}\E\sqbk{\ltwonorm{\rvx_{k}^{i}(i, j) - \rvx_{k}^{i}}^2} + (n-(i-1))\E\sqbk{\ltwonorm{\nabla f_{\pi_k^{i}}(\rvx_k^i)}^2}\\
    =& 2 \sum_{r=1}^{n}\E\sqbk{\ltwonorm{\nabla f_r(\rvx_k^i)}^2} - (n-(i-1))\E\sqbk{\ltwonorm{\nabla f_{\pi_k^{i}}(\rvx_k^i)}^2}\\
    &\qquad + 2L^2 \sum_{j<i}\E\sqbk{\ltwonorm{\rvx_{k}^{i}(i, j) - \rvx_{k}^{i}}^2}
\end{align*}
Since $i-1\leq n-1$ and $Q_{k}^{i}\geq 0$, then
\begin{align*}
    \E\sqbk{\ltwonorm{\nabla f_{\pi_k^{i}}(\rvx_k^i)}^2} \leq& \frac{2}{n}\sum_{r=1}^{n}\E\sqbk{\ltwonorm{\nabla f_r(\rvx_k^i)}^2} + \frac{2L^2}{n} \sum_{j<i}\E\sqbk{\ltwonorm{\rvx_{k}^{i}(i, j) - \rvx_{k}^{i}}^2}\\
    \leq& \frac{2}{n} \sum_{r=1}^{n}\E\sqbk{\ltwonorm{\nabla f_r(\rvx_k^i)}^2} + \frac{8\eta^2 L^2}{n}\sum_{j<i} \E\sqbk{\ltwonorm{\nabla f_{\pi_k^{j}}(\rvx_k^j)}^2}\\
    \leq& 8L\E\sqbk{f(\rvx_{k}^{i}) - f(\rvx_{*})} + 4\sigma_{*}^{2} + \frac{8\eta^2 L^2}{n}\sum_{j<i} \E\sqbk{\ltwonorm{\nabla f_{\pi_k^{j}}(\rvx_k^j)}^2}
\end{align*}
Where the last inequality holds by Lemma~\ref{lem:non_expansiveness}, we complete the proof of~\cref{equ:4.11}.

\subsection{Proof of Lemma~\ref{lem:exact_comparator_bias_identify}}

\paragraph{Proof of~\cref{equ:4.16}.} Fix $(p,q)\in\cI_N$ and define the history before this position and the set of remaining component indices by
\[
\cF_p^q:=\sigma(\pi_1,\ldots,\pi_{p-1},\pi_p^1,\ldots,\pi_p^{q-1}),
\qquad
\mathcal R_p^q:=[n]\setminus\set{\pi_p^1,\ldots,\pi_p^{q-1}}.
\]
The initialization is deterministic, and $\rvx_p^q$ depends only on the indices revealed before $(p,q)$, so $\rvx_p^q$ is $\cF_p^q$-measurable. Also, $|\mathcal R_p^q|=n-q+1$.

First consider $(k,i)\in\cI_N(p,q)$ with $k=p$, so that $i\geq q$. Conditional on $\cF_p^q$, the unrevealed suffix of $\pi_p$ is a uniform permutation of $\mathcal R_p^q$. In particular, both $\pi_p^i$ and $\pi_p^q$ are conditionally uniform on this set. Hence
\[
\E\sqbk{f_{\pi_p^i}(\rvx_p^q)\mid\cF_p^q}=\frac{1}{n-q+1}\sum_{r\in\mathcal R_p^q}f_r(\rvx_p^q)=\E\sqbk{f_{\pi_p^q}(\rvx_p^q)\mid\cF_p^q}.
\]
This equality uses only the identical conditional marginal distributions of the two indices. Taking total expectations and using the definition of $\beta_p^q$ gives
\[
\E\sqbk{f_{\pi_p^i}(\rvx_p^q)-f(\rvx_p^q)}
=\E\sqbk{f_{\pi_p^q}(\rvx_p^q)-f(\rvx_p^q)}
=-\beta_p^q.
\]

Next consider $k>p$. The fresh permutation $\pi_k$ is independent of $\cF_p^q$, and each of its positions is uniform on $[n]$. Therefore
\[
\E\sqbk{f_{\pi_k^i}(\rvx_p^q)\mid\cF_p^q}
=\frac1n\sum_{r=1}^n f_r(\rvx_p^q)
=f(\rvx_p^q),
\]
so $\E[f_{\pi_k^i}(\rvx_p^q)-f(\rvx_p^q)]=0$. Combining the two cases proves~\cref{equ:4.16}: the comparison-point bias equals $-\beta_p^q$ at every remaining position of epoch $p$ and vanishes in every later epoch.

\paragraph{Proof of~\cref{equ:4.17}.} We first establish an inequality that holds for each realization of the permutations. For any comparison point $\rvx\in\R^d$, convexity gives
\[
f_{\pi_k^i}(\rvx_k^i)-f_{\pi_k^i}(\rvx)
\leq\langle\nabla f_{\pi_k^i}(\rvx_k^i),\rvx_k^i-\rvx\rangle.
\]
Expanding the squared distance after the RR-SGD update yields
\begin{align*}
\ltwonorm{\rvx_k^{i+1}-\rvx}^2
={}&\ltwonorm{\rvx_k^i-\rvx}^2
-2\eta\langle\nabla f_{\pi_k^i}(\rvx_k^i),\rvx_k^i-\rvx\rangle
+\eta^2\ltwonorm{\nabla f_{\pi_k^i}(\rvx_k^i)}^2.
\end{align*}
Combining these two relations and dividing by $2\eta$ gives
\[
f_{\pi_k^i}(\rvx_k^i)-f_{\pi_k^i}(\rvx)
\leq\frac{\ltwonorm{\rvx_k^i-\rvx}^2-\ltwonorm{\rvx_k^{i+1}-\rvx}^2}{2\eta}
+\frac{\eta}{2}\ltwonorm{\nabla f_{\pi_k^i}(\rvx_k^i)}^2.
\]
Recall that $N=\tau(k_N,i_N)$. Summing over the consecutive positions in $\cI_N(p,q)$ telescopes the distance terms, including across epoch boundaries because $\rvx_k^{n+1}=\rvx_{k+1}^1$. Thus
\begin{equation}
\begin{aligned}
&\sum_{(k,i)\in\cI_N(p,q)}
\sqbk{f_{\pi_k^i}(\rvx_k^i)-f_{\pi_k^i}(\rvx)}\\
&\quad\leq\frac{\ltwonorm{\rvx_p^q-\rvx}^2-\ltwonorm{\rvx_{k_N}^{i_N+1}-\rvx}^2}{2\eta}
+\frac{\eta}{2}\sum_{(k,i)\in\cI_N(p,q)}\ltwonorm{\nabla f_{\pi_k^i}(\rvx_k^i)}^2\\
&\quad\leq\frac{\ltwonorm{\rvx_p^q-\rvx}^2}{2\eta}
+\frac{\eta}{2}\sum_{(k,i)\in\cI_N(p,q)}\ltwonorm{\nabla f_{\pi_k^i}(\rvx_k^i)}^2.
\end{aligned}
\label{equ:4.19}
\end{equation}
The last step discards a nonpositive terminal-distance term. If needed, the successor $\rvx_{k_N}^{i_N+1}$ is defined by one additional update for the analysis.

Since~\cref{equ:4.19} holds for every realized trajectory and every comparison point, we may substitute the random point $\rvx=\rvx_p^q$ before taking expectations. No independence between this point and the subsequent sampled gradients is needed. With this choice, the initial-distance term is zero. The definitions of $\Delta_k^i$ and $\beta_k^i$, together with~\cref{equ:4.16}, give
\begin{align*}
\E\sqbk{f_{\pi_k^i}(\rvx_k^i)}
&=f(\rvx_*)+\Delta_k^i-\beta_k^i,\\
\E\sqbk{f_{\pi_k^i}(\rvx_p^q)}
&=f(\rvx_*)+\Delta_p^q-\beta_p^q\mathbf{1}_{\set{k=p}}.
\end{align*}
Consequently, each expected summand on the left-hand side of~\cref{equ:4.19} is
\[
\E\sqbk{f_{\pi_k^i}(\rvx_k^i)-f_{\pi_k^i}(\rvx_p^q)}
=\Delta_k^i-\Delta_p^q-\beta_k^i
+\beta_p^q\mathbf{1}_{\set{k=p}}.
\]
The positions belonging to epoch $p$ run from global index $\tau(p,q)$ through $e_{N,p}^q=\min\set{np,N}$, inclusive. Their number is therefore
\[
\sum_{(k,i)\in\cI_N(p,q)}\mathbf{1}_{\set{k=p}}
=e_{N,p}^q-\tau(p,q)+1.
\]
Taking expectations in~\cref{equ:4.19} now yields
\[
\sum_{(k,i)\in\cI_N(p,q)}(\Delta_k^i-\Delta_p^q)
-\sum_{(k,i)\in\cI_N(p,q)}\beta_k^i
+\bigl(e_{N,p}^q-\tau(p,q)+1\bigr)\beta_p^q\leq\frac{\eta}{2}\sum_{(k,i)\in\cI_N(p,q)}
\E\sqbk{\ltwonorm{\nabla f_{\pi_k^i}(\rvx_k^i)}^2}.
\]
Rearranging proves~\cref{equ:4.17}, with the sign of the comparison-point bias retained.

\paragraph{Proof of~\cref{equ:4.18}.} In~\cref{equ:4.19}, take $(p,q)=(1,1)$ and $\rvx=\rvx_*$. Since $\rvx_*$ is fixed and $\pi_k^i$ is marginally uniform on $[n]$,
\[
\E\sqbk{f_{\pi_k^i}(\rvx_*)}
=\frac1n\sum_{r=1}^n f_r(\rvx_*)
=f(\rvx_*).
\]
It follows that
\[
\E\sqbk{f_{\pi_k^i}(\rvx_k^i)-f_{\pi_k^i}(\rvx_*)}
=\Delta_k^i-\beta_k^i.
\]
Taking expectations in~\cref{equ:4.19} and using $\ltwonorm{\rvx_1^1-\rvx_*}^2=D^2$, we obtain
\[
\sum_{(k,i)\in\cI_N}(\Delta_k^i-\beta_k^i)
\leq\frac{D^2}{2\eta}
+\frac{\eta}{2}\sum_{(k,i)\in\cI_N}
\E\sqbk{\ltwonorm{\nabla f_{\pi_k^i}(\rvx_k^i)}^2}.
\]
Moving the sampling-bias sum to the right-hand side proves~\cref{equ:4.18}.

\subsection{Proof of Lemma~\ref{lem:4.5}}
For simplicity, we use the notation where
\[
Q_{k}^{i} = \E\sqbk{\ltwonorm{\nabla f_{\pi_{k}^{i}}(\rvx_{k}^{i})}^2}
\]
For the real-valued sequence $\{\Delta_k^i:(k,i)\in\cI_N\}$, define the averaged suffix
\[
A_p^q:=
\frac{1}{N-\tau(p,q)+1}
\sum_{(k,i)\in\cI_N(p,q)}\Delta_k^i.
\]
we have
\[
A_{1}^{1} = \frac{1}{N}\sum_{(k, i)\in \cI_{N}}\Delta_k^i,\qquad A_{k_N}^{i_N}=\Delta_{k_N}^{i_N}=\Delta_{N}
\]
Since $N=\tau(k_N,i_N)$, we have $\cI_N(k_N,i_N)=\{(k_N,i_N)\}$. For $\tau(p,q)<N$, let $(p',q')\in\cI_N$ be the next position, so that $\tau(p',q')=\tau(p,q)+1$. Then
\begin{align*}
A_{p'}^{q'}-A_p^q
&=
\frac{1}{N-\tau(p',q')+1}
\sum_{(k,i)\in\cI_N(p',q')}\Delta_k^i
-
\frac{1}{N-\tau(p,q)+1}
\sum_{(k,i)\in\cI_N(p,q)}\Delta_k^i\\
&=
\frac{\sum_{(k,i)\in\cI_N(p,q)}\Delta_k^i-\Delta_p^q}
{N-\tau(p,q)}
-
\frac{\sum_{(k,i)\in\cI_N(p,q)}\Delta_k^i}
{N-\tau(p,q)+1}\\
&=
\frac{
\sum_{(k,i)\in\cI_N(p,q)}\Delta_k^i
-
\bigl(N-\tau(p,q)+1\bigr)\Delta_p^q
}{
\bigl(N-\tau(p,q)\bigr)\bigl(N-\tau(p,q)+1\bigr)
}\\
&=
\frac{
\sum_{(k,i)\in\cI_N(p,q)}
\bigl(\Delta_k^i-\Delta_p^q\bigr)
}{
\bigl(N-\tau(p,q)\bigr)\bigl(N-\tau(p,q)+1\bigr)
}.
\end{align*}
Summing over $A_{p'}^{q'} - A_{p}^{q}$ for all $\tau(p, q)< N$, since $(p', q')$ is the next position after $(p, q)$, it gives
\[
\sum_{(p, q)\in \cI_{N}, \tau(p, q)< N} (A_{p'}^{q'} - A_{p}^{q}) = A_{k_N}^{i_N} - A_{1}^{1} = \Delta_{N} -  \frac{1}{N} \sum_{(k, i)\in \cI_{N}} \Delta_{k}^{i}
\]
this further implies
\begin{align*}
    \Delta_{N} =& \frac{1}{N} \sum_{(k, i)\in \cI_{N}} \Delta_{k}^{i} + \sum_{(p, q)\in \cI_{N}, \tau(p, q)< N} (A_{p'}^{q'} - A_{p}^{q})\\
    =& \frac{1}{N}\sum_{(k,i)\in\cI_N}\Delta_k^i+
\sum_{\substack{(p,q)\in\cI_N\\\tau(p,q)<N}}
\frac{\sum_{(k,i)\in\cI_N(p,q)}(\Delta_k^i-\Delta_p^q)}
{(N-\tau(p,q))(N-\tau(p,q)+1)}.
\end{align*}
Define the weight coefficient
\[
c_{N,p}^{q}:=
\frac{1}{(N-\tau(p,q))(N-\tau(p,q)+1)},
\qquad \tau(p,q)<N.
\]
Apply~\cref{equ:4.17,equ:4.18}, then
\begin{align*}
\Delta_N\leq{}&
\frac{D^2}{2\eta N}
+\underbrace{\frac{\eta}{2N}\sum_{(k,i)\in\cI_N}Q_k^i
+\frac{\eta}{2}\sum_{\substack{(p,q)\in\cI_N\\\tau(p,q)<N}}
c_{N,p}^{q}\sum_{(k,i)\in\cI_N(p,q)}Q_k^i}_{\text{Term contains $Q_{k}^{i}$}}\\
&+\underbrace{\frac{1}{N}\sum_{(k,i)\in\cI_N}\beta_k^i+\sum_{\substack{(p,q)\in\cI_N\\\tau(p,q)<N}}
c_{N,p}^{q}\sum_{(k,i)\in\cI_N(p,q)}\beta_k^i-\sum_{\substack{(k,i)\in\cI_N\\\tau(k,i)<N}}
c_{N,k}^{i}\bigl(e_{N,k}^i-\tau(k,i)+1\bigr)\beta_k^i}_{\text{Term contains $\beta$}}.
\end{align*}
We collect the coefficients of $Q_k^i$ and $\beta_k^i$ respectively.

\paragraph{The coefficient of $Q_{k}^{i}$.} Exchange the order of summation, since for $(p, q)$ whenever $\tau(p, q)<N$ and $(k, i)\in \cI_{N}(p, q)$, it is equivalent to $\tau(p, q)\leq \min\set{\tau(k, i), N-1}$. Therefore
\begin{equation}
    \label{equ:Q_coefficient_reformulation}
\sum_{\substack{(p,q)\in\cI_N\\\tau(p,q)<N}}
c_{N,p}^{q}\sum_{(k,i)\in\cI_N(p,q)}Q_k^i
=
\sum_{(k,i)\in\cI_N}
\bracket{\sum_{\substack{(p,q)\in\cI_N\\
\tau(p,q)\leq\min\set{\tau(k,i),N-1}}}c_{N,p}^{q}}Q_k^i.    
\end{equation}
Thus, the gradient terms satisfy
\begin{align}
&\frac{\eta}{2N}\sum_{(k,i)\in\cI_N}Q_k^i
+\frac{\eta}{2}
\sum_{\substack{(p,q)\in\cI_N\\\tau(p,q)<N}}
c_{N,p}^{q}\sum_{(k,i)\in\cI_N(p,q)}Q_k^i\nonumber\\
&\qquad=
\frac{\eta}{2}\sum_{(k,i)\in\cI_N}
\bracket{\frac{1}{N}
+\sum_{\substack{(p,q)\in\cI_N\\
\tau(p,q)\leq\min\set{\tau(k,i),N-1}}}c_{N,p}^{q}}Q_k^i\nonumber\\
&\qquad=
\frac{\eta}{2}\sum_{(k,i)\in\cI_N}
\bracket{\frac{1}{N}
+\frac{1}{N-\min\set{\tau(k,i),N-1}}-\frac{1}{N}}Q_k^i\nonumber\\
&\qquad=
\frac{\eta}{2}\sum_{(k,i)\in\cI_N}w_{N,k}^{i}Q_k^i.\nonumber
\end{align}
with the definition of $w_{N,k}^{i} = \frac{1}{\max\{N-\tau(k, i), 1\}}$.

\paragraph{The coefficient of $\beta$.} Recall the term that contains $\beta$ is. 
\[
\underbrace{\frac{1}{N}\sum_{(k, i)\in \cI_{N}}\beta_{k}^{i} +\sum_{\substack{(p,q)\in\cI_N\\\tau(p,q)<N}}
c_{N,p}^{q}\sum_{(k,i)\in\cI_N(p,q)}\beta_k^i}_{\text{Term}_{1}}-\underbrace{\sum_{\substack{(k,i)\in\cI_N\\\tau(k,i)<N}}
c_{N,k}^{i}\bigl(e_{N,k}^i-\tau(k,i)+1\bigr)\beta_k^i}_{\text{Term}_{2}}.
\]
For the $\text{Term}_{1}$, similar to~\cref{equ:Q_coefficient_reformulation}, the coefficient of $\beta_k^i$ is
\begin{align}
\frac{1}{N}\sum_{(k,i)\in\cI_N}\beta_k^i
+\sum_{\substack{(p,q)\in\cI_N\\\tau(p,q)<N}}
c_{N,p}^{q}\sum_{(k,i)\in\cI_N(p,q)}\beta_k^i=&
\sum_{(k,i)\in\cI_N}
\bracket{\frac{1}{N}
+\sum_{\substack{(p,q)\in\cI_N\\
\tau(p,q)\leq\min\set{\tau(k,i),N-1}}}c_{N,p}^{q}}\beta_k^i\nonumber\\
=&
\sum_{(k,i)\in\cI_N}
\frac{\beta_k^i}{N-\min\set{\tau(k,i),N-1}}.\nonumber
\end{align}
For a fixed $(k,i)$, when $\tau(k,i)<N$, subtracting the $\text{Term}_{2}$, the coefficient of $\beta_k^i$ is
\begin{align*}
\frac{1}{N-\tau(k, i)} - c_{N, k}^{i}(e_{N, k}^{i} - \tau(k, i) + 1) =& \frac{1}{N - \tau(k, i)} - \frac{e_{N,k}^i-\tau(k,i)+1}
{(N-\tau(k,i))(N-\tau(k,i)+1)}\\
=& \frac{N-e_{N,k}^i}
{(N-\tau(k,i))(N-\tau(k,i)+1)}\\
=&v_{N,k}^{i}
\end{align*}
When $\tau(k,i)=N$, only coefficient in $\text{Term}_{1}$ remains, since the coefficient in $\text{Term}_{2}$ is zero. Thus, the coefficient of $\beta_k^i$ is $1=v_{N,k}^{i}$. Therefore
\[
\Delta_N
\leq
\frac{D^2}{2\eta N}
+\frac{\eta}{2}\sum_{(k,i)\in\cI_N}w_{N,k}^{i}Q_k^i
+\sum_{(k,i)\in\cI_N}v_{N,k}^{i}\beta_k^i,
\]
which proves~\cref{equ:4.20}.

At the same time, we also show two inequalities~\cref{equ:4.21,equ:4.22} for $v_{N,k}^{i}$ and $w_{N,k}^{i}$, which will be used in the proof of Proposition~\cref{prop:weighted_gradient_endpoint_inequality}. 
\begin{lemma}
    \label{lem:weight_relation}
    For every $(k, i), (p, q)\in \cI_{N}$, the weights $v_{N,k}^{i}$ and $w_{N,k}^{i}$ defined in Lemma~\ref{lem:4.5} satisfy
    \begin{align}
    0\leq v_{N,k}^{i}\leq& w_{N,k}^{i}\label{equ:4.21}\\
    \frac{1}{n}\sum_{\substack{(k,i)\in\cI_N\\\tau(k,i)<N,\ (p,q)\in\cP_k^i}}v_{N,k}^{i}\leq& w_{N,p}^{q}\label{equ:4.22}
\end{align}
\end{lemma}
\begin{proof}
    In proving~\cref{equ:4.21}, the definition of $e_{N,k}^i=\min\{nk, N\}$ gives $\tau(k,i)\leq e_{N,k}^i\leq N$. When $\tau(k,i)<N$,
\begin{align*}
0\leq v_{N,k}^{i}
&=\frac{N-e_{N,k}^i}
{(N-\tau(k,i))(N-\tau(k,i)+1)}\\
&\leq\frac{1}{N-\tau(k,i)+1}
\leq\frac{1}{N-\tau(k,i)}
=w_{N,k}^{i}.
\end{align*}
When $\tau(k,i)=N$, $v_{N,k}^{i}=w_{N,k}^{i}=1$. This proves~\cref{equ:4.21}.

For~\cref{equ:4.22}, fix $(p,q)\in\cI_N$. Recall the definition of $\cP_k^i=\{(k, j): 1\leq j < i\}$, the condition $(p,q)\in\cP_k^i$ is equivalent to $k=p$ and $q<i$. Hence
\[
\set{(k,i)\in\cI_N:(p,q)\in\cP_k^i}
=
\set{(p, j): q < j\leq n, \tau(p, j)\leq N}
=
\set{(p,j):\tau(p,q)<\tau(p,j)\leq e_{N,p}^q}.
\]
where the last equality follows from $e_{N, p}^{q} = \min\{np, N\}$. For these positions, $e_{N,k}^i=e_{N,p}^q$. 

If $\tau(p,q)=N$, then there is not pair $(k, i)$ satisfying $\tau(p, q)< \tau(k, i)\leq N$, thus
\[
\frac{1}{n}
\sum_{\substack{(k,i)\in\cI_N\\
\tau(k,i)<N,\ (p,q)\in\cP_k^i}}v_{N,k}^{i}
=0\leq w_{N,p}^{q}.
\]
When $\tau(p,q)<N$.

\paragraph{Case 1: $e_{N,p}^q<N$.}
All relevant positions satisfy $\tau(k,i)<N$, and $e_{N,p}^q=np$. Thus
\begin{align}
\sum_{\substack{(k,i)\in\cI_N\\
\tau(k,i)<N,\ (p,q)\in\cP_k^i}}v_{N,k}^{i}=&
\sum_{j=q+1}^{n}
\frac{N-e_{N,p}^q}
{(N-\tau(p,j))(N-\tau(p,j)+1)}\nonumber\\
=&
(N-e_{N,p}^q)\sum_{j=q+1}^{n}
\bracket{\frac{1}{N-\tau(p,j)}
-\frac{1}{N-\tau(p,j)+1}}\nonumber\\
=&
(N-e_{N,p}^q)
\bracket{\frac{1}{N-e_{N,p}^q}
-\frac{1}{N-\tau(p,q)}}\nonumber\\
=&
\frac{e_{N,p}^q-\tau(p,q)}{N-\tau(p,q)}.\nonumber
\end{align}
When $e_{N,p}^q=\tau(p,q)$, both sides are zero. Since the positions lie in the same epoch of length $n$,
\[
0\leq e_{N,p}^q-\tau(p,q)\leq n-1\leq n.
\]
Consequently,
\[
\frac{1}{n}
\sum_{\substack{(k,i)\in\cI_N\\
\tau(k,i)<N,\ (p,q)\in\cP_k^i}}v_{N,k}^{i}
=
\frac{e_{N,p}^q-\tau(p,q)}{n(N-\tau(p,q))}
\leq\frac{1}{N-\tau(p,q)}
=w_{N,p}^{q}.
\]

\paragraph{Case 2: $e_{N,p}^q=N$.}
For every relevant position with $\tau(p,q)<\tau(k,i)<N$, we have $e_{N,k}^i=N$, so
\[
v_{N,k}^{i}
=
\frac{N-e_{N,k}^i}
{(N-\tau(k,i))(N-\tau(k,i)+1)}
=0.
\]
The terminal position is excluded by the condition $\tau(k,i)<N$ in~\cref{equ:4.22}. Therefore
\[
\frac{1}{n}
\sum_{\substack{(k,i)\in\cI_N\\
\tau(k,i)<N,\ (p,q)\in\cP_k^i}}v_{N,k}^{i}
=0\leq w_{N,p}^{q}.
\]
Combining the two cases prove~\cref{equ:4.22}.
\end{proof}

\subsection{Proof of Proposition~\ref{prop:weighted_gradient_endpoint_inequality}}
\label{app:subsec:proof_of_prop:weighted_gradient_endpoint_inequality}
By~\cref{equ:4.10,equ:4.11}, for every $(k,i)\in\cI_N$,
\begin{align}
\beta_k^i
&\leq \frac{\eta}{2}Q_k^i
+\frac{3\eta}{n}\sum_{(p,q)\in\cP_k^i}Q_p^q,
\label{equ:4.14}\\
Q_k^i
&\leq 8L\Delta_k^i+4\sigma_{*}^2
+\frac{8\eta^2L^2}{n}\sum_{(p,q)\in\cP_k^i}Q_p^q.
\label{equ:4.15}
\end{align}
For the change in summation order below, we also need the block sum including the terminal position. For a fixed $(p,q)\in\cI_N$. If $e_{N,p}^q<N$, all relevant positions precede $N$, so~\cref{equ:4.22} applies. If $\tau(p,q)=N$, the sum is empty. If $e_{N,p}^q=N$ and $\tau(p,q)<N$, only the terminal position contributes, and
\[
\frac{1}{n}
\sum_{\substack{(k,i)\in\cI_N\\(p,q)\in\cP_k^i}}v_{N,k}^{i}
=
\frac{1}{n}
\leq\frac{1}{N-\tau(p,q)}
=w_{N,p}^{q},
\]
because $1\leq N-\tau(p,q)\leq n-1$. Hence, in every case,
\[
\frac{1}{n}
\sum_{\substack{(k,i)\in\cI_N\\(p,q)\in\cP_k^i}}v_{N,k}^{i}
\leq w_{N,p}^{q}.
\]
Since $v_{N,k}^{i}\geq0$, multiplying~\cref{equ:4.14} by $v_{N,k}^{i}$ and summing gives
\begin{align*}
\sum_{(k,i)\in\cI_N}v_{N,k}^{i}\beta_k^i
\leq{}&
\frac{\eta}{2}\sum_{(k,i)\in\cI_N}v_{N,k}^{i}Q_k^i
+\frac{3\eta}{n}\sum_{(k,i)\in\cI_N}v_{N,k}^{i}
\sum_{(p,q)\in\cP_k^i}Q_p^q\\
={}&
\frac{\eta}{2}\sum_{(k,i)\in\cI_N}v_{N,k}^{i}Q_k^i
+3\eta\sum_{(p,q)\in\cI_N}Q_p^q
\bracket{\frac{1}{n}
\sum_{\substack{(k,i)\in\cI_N\\(p,q)\in\cP_k^i}}v_{N,k}^{i}}\\
\leq{}&
\frac{\eta}{2}\sum_{(k,i)\in\cI_N}w_{N,k}^{i}Q_k^i
+3\eta\sum_{(p,q)\in\cI_N}w_{N,p}^{q}Q_p^q.
\end{align*}
The last inequality uses~\cref{equ:4.21} and the block-sum bound above. Combining this with~\cref{equ:4.20} proves~\cref{equ:4.25}.

\subsection{Proof of Lemma~\ref{lem:harmonic_weight_block_sum}}
Fix $(p,q)\in\cI_N$. If $\tau(p,q)=N$, the sum in~\cref{equ:4.26} is empty. Otherwise, let
\[
a:=N-\tau(p,q)\geq1,\qquad
m:=e_{N,p}^q-\tau(p,q)\leq\min\set{n-1,a}.
\]
We first bound the sum including the terminal position. Since $(p,q)\in\cP_k^i$ means $k=p$ and $q<i$, the definition of $w_{N,k}^{i}$ gives
\[
\frac{1}{n}
\sum_{\substack{(k,i)\in\cI_N\\(p,q)\in\cP_k^i}}w_{N,k}^{i}
=
\frac{1}{n}
\bracket{\sum_{j=1}^{\min\set{m,a-1}}\frac{1}{a-j}
+\mathbf{1}_{\set{m=a}}}.
\]
If $a\leq n$, then $a/n\leq1$, and
\[
\frac{a}{n}
\bracket{\sum_{j=1}^{\min\set{m,a-1}}\frac{1}{a-j}
+\mathbf{1}_{\set{m=a}}}
\leq\cH_{a-1}+1
\leq2\cH_N.
\]
If $a>n$, then $m\leq n-1<a$, so the terminal term is absent. For $1\leq j\leq n-1$, we have $a/(a-j)\leq n/(n-j)$. Consequently,
\[
\frac{a}{n}\sum_{j=1}^{m}\frac{1}{a-j}
\leq\sum_{j=1}^{n-1}\frac{1}{n-j}
=\cH_{n-1}
\leq\cH_N.
\]
Since $w_{N,p}^{q}=1/a$, these two cases yield
\[
\frac{1}{n}
\sum_{\substack{(k,i)\in\cI_N\\(p,q)\in\cP_k^i}}w_{N,k}^{i}
\leq2\cH_Nw_{N,p}^{q}.
\]
This proves~\cref{equ:4.26}, including the contribution at the terminal position.

\subsection{Proof of Lemma~\ref{lem:all_time_last_iterate_bound}}
We argue by induction over the global positions. At $(k,i)=(1,1)$, smoothness and $\nabla f(\rvx_{*})=0$ give
\[
\Delta_1^1 = 
\E\sqbk{f(\rvx_1^1)-f(\rvx_{*})}
\leq\frac{L}{2}D^2
\leq\frac{D^2}{\eta}.
\]
For a fixed $N=\tau(k_N,i_N)\geq2$, suppose the conclusion holds at all earlier positions:
\[
\Delta_k^i
\leq\frac{D^2}{\eta\tau(k,i)}
+128\eta\sigma_{*}^2\cH_{\tau(k,i)},
\qquad (k,i)\in\cI_N,\quad \tau(k,i)<N.
\]
Since $w_{N,k_N}^{i_N}=1$ and $\Delta_N=\Delta_{k_N}^{i_N}$,
\[
V_N
=
\Delta_N
+\sum_{\substack{(k,i)\in\cI_N\\\tau(k,i)<N}}
\frac{\Delta_k^i}{N-\tau(k,i)}.
\]
At the same time, we have the harmonic type bound:
\begin{align*}
\sum_{\substack{(k,i)\in\cI_N\\\tau(k,i)<N}}
\frac{1}{\tau(k,i)(N-\tau(k,i))}
&=\frac{2\cH_{N-1}}{N},\\
\sum_{t=1}^{N-1}
\frac{\cH_{t}}{N-t}
&\leq \cH_{N}\sum_{t=1}^{N-1}\frac{1}{N-t}\leq\cH_N\cH_{N-1},
\end{align*}
respectively. By the induction hypothesis,
\begin{align}
V_N\leq{}&
\Delta_N
+\frac{D^2}{\eta}
\sum_{\substack{(k,i)\in\cI_N\\\tau(k,i)<N}}
\frac{1}{\tau(k,i)(N-\tau(k,i))}+128\eta\sigma_{*}^2
\sum_{\substack{(k,i)\in\cI_N\\\tau(k,i)<N}}
\frac{\cH_{\tau(k,i)}}{N-\tau(k,i)}\nonumber\\
\leq{}&
\Delta_N+\frac{2D^2\cH_{N-1}}{\eta N}
+128\eta\sigma_{*}^2
\sum_{\substack{(k,i)\in\cI_N\\\tau(k,i)<N}}
\frac{\cH_N}{N-\tau(k,i)}\nonumber\\
\leq{}&
\Delta_N+\frac{2D^2\cH_{N-1}}{\eta N}
+128\eta\sigma_{*}^2\cH_N\cH_{N-1}.
\label{equ:weighted_iteration}
\end{align}
Substituting~\cref{equ:weighted_iteration} into~\cref{equ:4.30} gives
\begin{align*}
\Delta_N\leq{}&
\frac{D^2}{2\eta N}+64\eta L V_N
+32\eta\sigma_{*}^2(\cH_{N-1} + 1)\\
\leq{}&
\frac{D^2}{2\eta N}+64\eta L\Delta_N
+128\eta L\frac{D^2\cH_{N-1}}{\eta N}+8192\eta L\eta\sigma_{*}^2\cH_N\cH_{N-1}
+32\eta\sigma_{*}^2(\cH_{N-1} + 1).
\end{align*}
Since $\eta L \cH_{N}\leq 1/512$, we obtain
\[
\frac{1}{2}+128\eta L\cH_{N-1}\leq\frac{3}{4},
\qquad
8192\eta L\cH_N\cH_{N-1}
+32(\cH_{N-1}+1)\leq80\cH_N.
\]
Moving $64\eta L \Delta_{N}$ to the left hand side
\begin{align}
(1-64\eta L)\Delta_N\leq{}&
\frac{D^2}{\eta N}
\bracket{\frac{1}{2}+128\eta L\cH_{N-1}}+\bracket{8192\eta L\cH_N\cH_{N-1}
+32(\cH_{N-1}+1)}\eta\sigma_{*}^2\nonumber\\
\leq{}&
\frac{3D^2}{4\eta N}
+80\eta\sigma_{*}^2\cH_N.
\label{equ:4.32}
\end{align}
Also, $1-64\eta L\geq7/8$. Therefore
\[
\Delta_N
\leq\frac{6}{7}\frac{D^2}{\eta N}
+\frac{640}{7}\eta\sigma_{*}^2\cH_N
\leq\frac{D^2}{\eta N}
+128\eta\sigma_{*}^2\cH_N,
\]
which completes the induction.

\subsection{Proof of Lemma~\ref{lem:within_epoch_comparison}}
\label{subsec:app:within_epoch_comparison}
When $n=1$, we have $\sigma_*^2=0$, and RR-SGD degenerates to gradient descent. Since $\eta L\leq1/2$, its objective does not increase in a single update, which proves the claim for $i=1,2$. We therefore assume $n\geq2$.

Fix an epoch $k$ and condition on its starting point $\rvy=\rvx_k^1$. The fresh permutation $\pi_k$ is independent of $\rvy$; write $\E_k[\cdot]=\E[\cdot\mid\rvx_k^1]$ and $G_{\rvy}=f(\rvy)-f(\rvx_*)$. Define
\[
a_j:=\nabla f_j(\rvx_*),\qquad
b_j(\rvy):=\nabla f_j(\rvy)-\nabla f_j(\rvx_*).
\]
Since $\nabla f(\rvx_*)=0$, we have $\sum_{j=1}^n a_j=0$ and $n^{-1}\sum_{j=1}^n\|a_j\|^2=\sigma_*^2$. Co-coercivity, averaged over the components, also gives
\[
\frac1n\sum_{j=1}^n\|b_j(\rvy)\|^2
\leq \frac{2L}{n}\sum_{j=1}^n
\bracket{f_j(\rvy)-f_j(\rvx_*)-\langle a_j,\rvy-\rvx_*\rangle}=2LG_{\rvy}.
\]
For $0\leq s\leq n$, the variance formula for sampling without replacement yields
\[
\E_k\left\|\sum_{r=1}^s a_{\pi_k^r}\right\|^2
=\frac{s(n-s)}{n-1}\sigma_*^2
\leq n\sigma_*^2.
\]
Indeed, for distinct positions $r,u$, the expected inner product is
$\E_k\langle a_{\pi_k^r},a_{\pi_k^u}\rangle=-\sigma_*^2/(n-1)$.
By Cauchy--Schwarz and the uniform marginal distribution at each position,
\[
\E_k\left\|\sum_{r=1}^s b_{\pi_k^r}(\rvy)\right\|^2
\leq s\sum_{r=1}^s\E_k\|b_{\pi_k^r}(\rvy)\|^2
\leq2s^2LG_{\rvy}.
\]
Consequently, the sum of gradients evaluated at the epoch's starting point satisfies
\[
\E_k\left\|\sum_{r=1}^s\nabla f_{\pi_k^r}(\rvy)\right\|^2
\leq4n^2LG_{\rvy}+2n\sigma_*^2=:A_{\rvy}^2.
\]

Let $d_s=(\E_k\|\rvx_k^{s+1}-\rvy\|^2)^{1/2}$, so that $d_0=0$. The update rule gives
\[
\rvx_k^{s+1}-\rvy
=-\eta\sum_{r=1}^s\nabla f_{\pi_k^r}(\rvy)-\eta\sum_{r=1}^s
\bracket{\nabla f_{\pi_k^r}(\rvx_k^r)-\nabla f_{\pi_k^r}(\rvy)}.
\]
Taking conditional $L^2$ norms and applying smoothness yields
\[
d_s\leq\eta A_{\rvy}+\eta L\sum_{r=0}^{s-1}d_r.
\]
An induction therefore gives, for $1\leq s\leq n$,
\[
d_s\leq\eta A_{\rvy}(1+\eta L)^{s-1}
\leq\eta A_{\rvy}e^{\eta nL}.
\]
Since $\eta nL\leq1/2$, this implies
\[
\E_k\|\rvx_k^{s+1}-\rvy\|^2
\leq e\eta^2\bracket{4n^2LG_{\rvy}+2n\sigma_*^2}.
\]
Finally, smoothness of $f$ and $\|\nabla f(\rvy)\|^2\leq2LG_{\rvy}$ imply, for any $\rvx$,
\[
f(\rvx)-f(\rvx_*)
\leq2G_{\rvy}+L\|\rvx-\rvy\|^2.
\]
Combining the preceding bounds, for $i=s+1$ we obtain
\[
\E_k[f(\rvx_k^i)-f(\rvx_*)]\leq\bracket{2+4e(\eta nL)^2}G_{\rvy}
+2e\eta^2nL\sigma_*^2\leq6G_{\rvy}+8\eta^2nL\sigma_*^2.
\]
The case $i=1$ is immediate. Taking total expectations proves the claim for every $1\leq i\leq n+1$.

\subsection{Proof of Theorem~\ref{thm:main}}
Fix a pair $(k,i)$ chosen before the run, with $1\leq i\leq n$ and $2\leq N=\tau(k,i)\leq nK+1$, and let $\Lambda=1+\log(nK)$. We retain explicit constants throughout the proof. The assumed stepsize satisfies
\[
\eta L\cH_N
\leq\eta L\cH_{nK+1}
\leq2\eta L\Lambda
\leq\frac1{512}.
\]
Thus Lemma~\ref{lem:all_time_last_iterate_bound} gives, at every such fixed position,
\begin{equation}
\Delta_k^i
\leq\frac{D^2}{\eta N}+256\Lambda\eta\sigma_*^2.
\label{equ:app_all_positions_large_step}
\end{equation}
We combine this bound with the small-step endpoint result using
\[
\eta_0:=\frac{1}{4nL\sqrt{2(1+\log K)}}.
\]

\paragraph{Small stepsizes: $\eta\leq\eta_0$.}
Here
\[
\eta nL\leq1/(4\sqrt2)<1/2,
\]
so Lemma~\ref{lem:within_epoch_comparison} applies. We first handle positions in the first two epochs. For $k=1$, smoothness at the minimizer gives $\Delta_1^1\leq LD^2/2$, and hence
\begin{align*}
\Delta_1^i
&\leq6\Delta_1^1+8\eta^2nL\sigma_*^2\\
&\leq3LD^2+8\eta^2nL\sigma_*^2\\
&\leq\frac{3D^2}{2\eta N}+8\eta^2nL\sigma_*^2.
\end{align*}
The last inequality uses $N=i\leq n$ and $\eta LN\leq\eta Ln\leq1/2$.
For $k=2$, use $\rvx_2^1=\rvx_1^{n+1}$ and apply the same lemma twice:
\begin{align*}
\Delta_2^i
&\leq6\Delta_2^1+8\eta^2nL\sigma_*^2\\
&\leq36\Delta_1^1+56\eta^2nL\sigma_*^2\\
&\leq18LD^2+56\eta^2nL\sigma_*^2\\
&\leq\frac{18D^2}{\eta N}+56\eta^2nL\sigma_*^2.
\end{align*}
Here $N=n+i\leq2n$, so $\eta LN\leq2\eta nL\leq1$.

For $k\geq3$, set $m=k-1$. Then $2\leq m\leq K$, and
\[
\eta\leq\eta_0
\leq\frac1{4nL\sqrt{2(1+\log m)}}.
\]
The random-permutation case in the proof of Theorem~C.1 of~\citet{liu2024last} gives an explicit version of Lemma~\ref{lem:theorem_c_1}. With no regularizer and $L=\max_r L_r$, its bound after $m$ completed epochs is
\[
\Delta_k^1
\leq\frac{D^2}{\eta nm}
+\frac{16}{3}\eta^2nL\sigma_*^2(1+\log m).
\]
Applying Lemma~\ref{lem:within_epoch_comparison} therefore gives
\begin{align*}
\Delta_k^i
&\leq6\Delta_k^1+8\eta^2nL\sigma_*^2\\
&\leq\frac{6D^2}{\eta nm}
+\bigl[32(1+\log m)+8\bigr]\eta^2nL\sigma_*^2\\
&\leq\frac{12D^2}{\eta N}
+40\Lambda\eta^2nL\sigma_*^2.
\end{align*}
The last step uses $N=nm+i\leq2nm$, $1+\log m\leq\Lambda$, and $\Lambda\geq1$. Combining the three cases $k=1$, $k=2$, and $k\geq3$, we obtain
\begin{equation}
\Delta_k^i
\leq\frac{18D^2}{\eta N}
+56\Lambda\eta^2nL\sigma_*^2
\leq\frac{18D^2}{\eta N}
+56\Lambda^{3/2}\eta^2nL\sigma_*^2.
\label{equ:app_small_step_explicit}
\end{equation}
Since $\eta nL<1$, the noise term in the minimum of Theorem~\ref{thm:main} is $\eta^2nL\sigma_*^2$ in this regime.

\paragraph{Intermediate stepsizes: $\eta>\eta_0$ and $\eta nL<1$.}
In this case,
\[
\frac1{\eta nL}<4\sqrt{2(1+\log K)}\leq4\sqrt{2\Lambda}.
\]
Consequently,
\[
256\Lambda\eta\sigma_*^2
=\frac{256\Lambda}{\eta nL}\eta^2nL\sigma_*^2
\leq1024\sqrt2\,\Lambda^{3/2}\eta^2nL\sigma_*^2.
\]
Substitution into~\cref{equ:app_all_positions_large_step} yields
\begin{equation}
\Delta_k^i
\leq\frac{D^2}{\eta N}
+1024\sqrt2\,\Lambda^{3/2}\eta^2nL\sigma_*^2
\leq\frac{D^2}{\eta N}
+2048\Lambda^{3/2}\eta^2nL\sigma_*^2.
\label{equ:app_intermediate_step_explicit}
\end{equation}
Again, the minimum equals $\eta^2nL\sigma_*^2$ because $\eta nL<1$.

\paragraph{Larger stepsizes: $\eta nL\geq1$.}
Now the minimum equals $\eta\sigma_*^2$. By~\cref{equ:app_all_positions_large_step} and $\Lambda\leq\Lambda^{3/2}$,
\begin{equation}
\Delta_k^i
\leq\frac{D^2}{\eta N}+256\Lambda\eta\sigma_*^2
\leq\frac{D^2}{\eta N}+256\Lambda^{3/2}\eta\sigma_*^2.
\label{equ:app_larger_step_explicit}
\end{equation}
Combining~\cref{equ:app_small_step_explicit,equ:app_intermediate_step_explicit,equ:app_larger_step_explicit}, and using $18\geq1$ and $2048\geq\max\{56,256\}$, proves the explicit bound
\[
\Delta_k^i
\leq\frac{18D^2}{\eta\tau(k,i)}
+2048\Lambda^{3/2}\min\set{\eta\sigma_*^2,\eta^2nL\sigma_*^2}
\]
for every pair fixed before the run in the stated range, including
$(k,i)=(K+1,1)$. This establishes Theorem~\ref{thm:main}.

\subsection{Proof of Corollary~\ref{cor:tuned}}
\label{subsec:app:proof_cor_tuned}
Let $T=nK$, $\Lambda=1+\log T$, and $M=1/(1024L\Lambda)$. Applying Theorem~\ref{thm:main} at $(K+1,1)$ and using $\tau(K+1,1)=T+1\geq T$ gives
\[
\E[f(\rvx_{K+1}^1)-f(\rvx_*)]
\lesssim \frac{D^2}{\eta T}
+\Lambda^{3/2}\sigma_*^2\min\{\eta,nL\eta^2\}.
\]
If $\sigma_*=0$, choosing $\eta=M$ gives the claim immediately. For $\sigma_*>0$, the two candidates in Corollary~\ref{cor:tuned} are obtained by balancing the initial-distance term in Theorem~\ref{thm:main} against each of its two noise expressions:
\[
\frac{D^2}{T\eta_{\mathrm{lin}}}
=\Lambda^{3/2}\sigma_*^2\eta_{\mathrm{lin}},
\qquad
\frac{D^2}{T\eta_{\mathrm{quad}}}
=\Lambda^{3/2}\sigma_*^2 nL\eta_{\mathrm{quad}}^2.
\]
Taking $u=\max\{\eta_{\mathrm{lin}},\eta_{\mathrm{quad}}\}$ selects the smaller of the two balanced initial-distance terms. Since $u$ equals one of these candidates, the corresponding equality above implies
\[
\Lambda^{3/2}\sigma_*^2\min\{u,nLu^2\}
\leq \frac{D^2}{Tu}.
\]
The prescribed choice $\eta=\min\{M,u\}$ enforces the stepsize restriction in Theorem~\ref{thm:main}. The noise term is nondecreasing in the stepsize, and $1/\eta\leq1/M+1/u$. Hence
\begin{align*}
\E[f(\rvx_{K+1}^1)-f(\rvx_*)]
&\lesssim \frac{D^2}{MT}+\frac{2D^2}{Tu}\\
&\lesssim \frac{\Lambda LD^2}{nK}
+\min\left\{\frac{\Lambda^{3/4}\sigma_*D}{\sqrt{nK}},
\Lambda^{1/2}\left(\frac{L\sigma_*^2D^4}{nK^2}\right)^{1/3}\right\}.
\end{align*}
Hiding the logarithmic factors yields Corollary~\ref{cor:tuned}.

\end{document}